\documentclass[final]{siamltex}
\usepackage{graphicx}
\usepackage{color}
\usepackage{bm}
\usepackage{amsfonts}
\usepackage{epsfig}
\usepackage{epstopdf}

\newfont{\bb}{msbm10}

\newcommand{\tr}{^{\sf T}}
\newcommand{\C}[1]{{\cal {#1}}}
\newcommand{\g}[1]{\mbox{\boldmath $#1$}}
\newcommand{\m}[1]{{\boldsymbol{#1}}}
\newcommand{\LAMBDA}{{\mbox{\boldmath $\lambda$}}}
\newcommand{\DELTA}{{\mbox{\boldmath $\delta$}}}
\renewcommand{\tilde}{\widetilde}
\renewcommand{\hat}{\widehat}
\renewcommand{\bar}{\overline}

\newtheorem{remark}{Remark}[section]
\newtheorem{assumption}{Assumption}[section]

\begin{document}
\title{Convergence on a symmetric accelerated stochastic ADMM with larger stepsizes
\thanks{
This research was   supported by  the National  Natural Science Foundation of China (12001430, 11625105, 72071158),  the Fundamental Research Funds for the Central Universities (G2020KY05203),    the China Postdoctoral Science Foundation (2020M683545), the Foundation of National Key Laboratory of Science and Technology on Aerodynamic Design and Research (614220119040101)    and   the USA National Science Foundation
(1819161, {2110722}).}
}

\author{
    Jianchao Bai\thanks{{\tt jianchaobai@nwpu.edu.cn},
         {School of Mathematics and Statistics},  the MIIT Key Laboratory of Dynamics and Control of Complex Systems,  Northwestern Polytechnical
        University, Xi'an  710129,     China.}
\and
        Deren Han\thanks{{\tt handr@buaa.edu.cn},
        LMIB of the Ministry of Education, School of Mathematical Sciences,
        Beihang University, Beijing  100191,   China}
\and
    Hao Sun\thanks{{\tt hsun@nwpu.edu.cn},
        {School of Mathematics and Statistics,}  Northwestern Polytechnical
        University, Xi'an  710129,    China.}
\and
    Hongchao Zhang\thanks{{\tt hozhang@math.lsu.edu},
        Department of Mathematics,
        Louisiana State University, Baton Rouge, LA 70803-4918.}
}
\maketitle

\begin{abstract}
In this paper,  we develop a symmetric accelerated stochastic
Alternating Direction Method of Multipliers (SAS-ADMM) for solving separable
convex optimization problems with linear constraints.
The objective  function  is the sum of a possibly  nonsmooth convex function and an
average function of many smooth convex functions.
Our proposed algorithm combines both ideas of ADMM and the techniques of
accelerated stochastic gradient methods  possibly with  variance reduction to solve
the smooth subproblem.
One main feature of SAS-ADMM {is} that its dual variable is symmetrically updated
after each update of the separated primal variable,
which would allow a more flexible and larger convergence region
of the dual variable compared with that of standard deterministic or stochastic ADMM.
This new stochastic optimization algorithm is shown to have ergodic converge
in expectation with  $\C{O}(1/T)$ convergence rate, where $T$ is the number of outer iterations.
Our preliminary   experiments indicate the proposed algorithm is very effective
for solving separable optimization problems from big-data applications.
Finally, 3-block extensions of the algorithm and its variant of an accelerated stochastic
augmented Lagrangian method are discussed in the appendix.

\end{abstract}



\pagestyle{myheadings}
\thispagestyle{plain}

\section{Introduction}
We consider the following structured composite convex optimization problem with
linear equality constraints:
\begin{equation} \label{P}
\min \{ f(\m{x}) + g(\m{y}) \mid
 \m{x} \in \C{X}, \; \m{y} \in \C{Y}, \; A\m{x} + B\m{y} = \m{b}\},
\end{equation}
where  $\C{X} \subset \mathbb{R}^{n_1}, \C{Y} \subset \mathbb{R}^{n_2} $ are  closed  convex  subsets,
 $A\in \mathbb{R}^{n \times n_1}$, $B\in \mathbb{R}^{n \times n_2}$,   $\m{b}\in \mathbb{R}^{n}$
are  given, $g: \C{Y} \to \mathbb{R} \cup \{+\infty\} $ is a convex
but possibly  nonsmooth function,  and $f$ is an average of $N$ real-valued  convex functions:
\[
f(\m{x})=\frac{1}{N}\sum\limits_{j=1}^{N}f_{j}(\m{x}).
\]
We assume that each $f_j$ defined on an open set
containing $\C{X}$ is Lipschitz continuously differentiable on $\mathcal{X}$.
Problem (\ref{P}) is also referred as regularized empirical risk minimization
in big-data applications \cite{Ouyang13,ZhaoLiZhou15}, including classification and regression models
in machine learning, where $N$ denotes the  sample size and $f_{j} $ corresponds to the empirical loss.
A major difficulty for solving (\ref{P}) is that the sample size $N$ can be very large
such that it is often computationally prohibitive to evaluate either the full function value or
the gradient of $f$ at each iteration of an algorithm.
Hence, it is essential for an effective algorithm, e.g., a stochastic gradient method,
to explore the summation structure of $f$ in the objective function.

The  augmented Lagrangian function of (\ref{P}) is
\begin{equation} \label{aug-Lagr}
\C{L}_\beta\left( \m{x}, \m{y},\LAMBDA\right) =
\C{L}\left( \m{x}, \m{y},\LAMBDA\right)+\frac{\beta}{2}
\left\| A\m{x}+B\m{y}-\m{b}\right\|^2,
\end{equation}
where  $\beta > 0$ is a  penalty parameter, $\LAMBDA$ is the  Lagrange multiplier and
the Lagrangian of (\ref{P}) is defined as
\begin{equation}\label{Lagr}
\C{L}\left(\m{x}, \m{y},\LAMBDA\right)
=f(\m{x})+g(\m{y})-\LAMBDA\tr
\left( A\m{x}+B\m{y}-\m{b} \right).
\end{equation}
Although the Augmented Lagrangian Method (ALM) can be applied to solve (\ref{P}),
it does not take advantage of  the separable structure of (\ref{P}).
As a splitting version of ALM, the standard
Alternating Direction Method of Multipliers (ADMM, \cite{GM75,GM76})
exploits the separable structure of the objective function and performs the following iterations:
\[
\left \{\begin{array}{l}
\m{x}^{k+1}\in\arg\min\limits_{\m{x}\in\C{X}} \;\C{L}_\beta(\m{x},\m{y}^k,\LAMBDA^k),\\
\m{y}^{k+1}\in\arg\min\limits_{\m{y}\in\C{Y}}\; \C{L}_\beta(\m{x}^{k+1},\m{y},\LAMBDA^k),\\
\LAMBDA^{k+1}=
\LAMBDA^k- s\beta\left( A\m{x}^{k+1}+B\m{y}^{k+1}-\m{b} \right),
\end{array}\right.
\]
where $s\in (0, \frac{1+\sqrt{5}}{2})$  {is}   the stepsize for updating the dual variable $\LAMBDA$.

 { If the Peaceman-Rachford Splitting
Method (PRSM, \cite{PR55})  is applied to the dual of (\ref{P}), then we obtain a  variation of ADMM,
whose iteration reads
\[
\left \{\begin{array}{lll}
\m{x}^{k+1}&\in&\arg\min\limits_{\m{x}\in\C{X}} \C{L}_\beta(\m{x},\m{y}^k,\LAMBDA^k),\\
\LAMBDA^{k+\frac{1}{2}}&=&
\LAMBDA^k-  \beta\left( A\m{x}^{k+1}+B\m{y}^k-\m{b} \right){,}\\
\m{y}^{k+1}&\in&\arg\min\limits_{\m{y}\in\C{Y}} \C{L}_\beta(\m{x}^{k+1},\m{y},\LAMBDA^{k+\frac{1}{2}}),\\
\LAMBDA^{k+1}&=&
\LAMBDA^{k+\frac{1}{2}}-  \beta\left( A\m{x}^{k+1}+B\m{y}^{k+1}-\m{b} \right).
\end{array}\right.
\]
PRSM is also called Symmetric ADMM (S-ADMM)  since the Lagrange multipliers are symmetrically updated twice in each loop. Note that both updates of dual variable in PRSM  use  the same constant  stepsize 1.
  Recently, Luo-Yang \cite{LYang20}  proposed a fast S-ADMM for  solving (\ref{P}) with only equality constraints, where the Nesterov's acceleration technique was applied for an additional update of
 $\LAMBDA^{k+1}$ and the $\m{y}$ variable was updated again by solving
\[ \min\limits_{\m{y}} g(\m{y})-(\hat{\LAMBDA}^{k+1})\tr B\m{y}~~ \textrm{ with } ~~\hat{\LAMBDA}^{k+1}= \LAMBDA^{k+1} +\frac{\theta^{k+1}(1-\theta^k)}{\theta^k}(\LAMBDA^{k+1}-\LAMBDA^k)
\]
and  $\theta^{k+1}= {2}/(k+1)$.
  }
Motivated from the ideas of enlarging the dual stepsize in \cite{HLWY14},
Gu, et al. \cite{GuJH15} proposed  a symmetric proximal ADMM whose dual variable is updated twice with different stepsizes.
{Meanwhile,}    {  the following extension of S-ADMM was developed by He, et al. \cite{HeMaYuan2016}:}
\begin{equation} \label{He-sy}
\left \{\begin{array}{lll}
\m{x}^{k+1}&\in&\arg\min\limits_{\m{x}\in\C{X}} \C{L}_\beta(\m{x},\m{y}^k,\LAMBDA^k),\\
\LAMBDA^{k+\frac{1}{2}}&=&
\LAMBDA^k- \tau\beta\left( A\m{x}^{k+1}+B\m{y}^k-\m{b} \right){,}\\
\m{y}^{k+1}&\in&\arg\min\limits_{\m{y}\in\C{Y}} \C{L}_\beta(\m{x}^{k+1},\m{y},\LAMBDA^{k+\frac{1}{2}}),\\
\LAMBDA^{k+1}&=&
\LAMBDA^{k+\frac{1}{2}}- s\beta\left( A\m{x}^{k+1}+B\m{y}^{k+1}-\m{b} \right),
\end{array}\right.
\end{equation}
where, {for convergence, the stepsize pair $(\tau,s)$ is required to belong to the following region}:
\[
\Delta_0=\left\{(\tau,s)\mid ~s\in  (0, (1+\sqrt{5})/2),~ \tau+s>0,~\tau\in(-1,1),~ \mid\tau\mid<1+s-s^2\right\}.
\]
Bai, et al. \cite{BLXZ2018} further designed a Generalized  Symmetric ADMM (GS-ADMM) for solving
a multi-block separable convex optimization  and enlarged the above   region $\Delta_0$  to
$\Delta$ defined in (\ref{domainK}).
{Numerical experiments show that symmetrically updating the dual variable in a more
flexible way often improves the algorithm performance  \cite{BLXZ2018,GuJH15,HeMaYuan2016}.
The  sublinear convergence rate of GS-ADMM in the nonergodic sense and its linear convergence rate
were shown in \cite{BXCL19}.
To our knowledge,} $\Delta$  is currently the largest convergence
region of the dual stepsizes for symmetric ADMM-type algorithms and has been used in the logarithmic-quadratic proximal based  ADMM for solving the 2-block problems \cite{WuL19} and the grouped multi-block problems \cite{BMS20}.

{Another line of developing ADMM is to apply relaxation techniques such as using
\[
\m{\chi}^{k+1}=\alpha A \m{x}^{k+1} + (1-\alpha)(\m{b}-B\m{y}^k)
\]
to replace $A\m{x}^{k+1}$ during the update of $\m{y}^{k+1}$ and $\LAMBDA^{k+1}$, where  $\alpha\in (0,2)$ is a relaxation factor.  This leads to the classical Generalized ADMM (G-ADMM, \cite{EcBe92}):
\[
\left \{\begin{array}{lll}
\m{x}^{k+1}&\in&\arg\min\limits_{\m{x}\in\C{X}} \C{L}_\beta(\m{x},\m{y}^k,\LAMBDA^k),\\
\m{y}^{k+1}&\in&\arg\min\limits_{\m{y}\in\C{Y}} g(\m{y})- (\LAMBDA^k)\tr B\m{y} +\frac{\beta}{2}
\left\| \m{\chi}^{k+1}+B\m{y}-\m{b}\right\|^2,\\
\LAMBDA^{k+1}&=&
\LAMBDA^k-  \beta\left( \m{\chi}^{k+1}+B\m{y}^{k+1}-\m{b} \right).
\end{array}\right.
\]
 Clearly,  G-ADMM with $\alpha=1$  would reduce to the standard ADMM with unit dual stepsize.
ADMM using some additional proximal terms in their subproblems is also called G-ADMM.
For instance, the $\m{x}$-subproblem in \cite{FHLY15} was proposed as
$  \min\limits_{\m{x}\in\C{X}} \C{L}_\beta(\m{x},\m{y}^k,\LAMBDA^k)+\frac{1}{2}\|\m{x}-\m{x}^k\|_G^2$,
 where $G$ is a symmetric positive definite matrix.
For more recent G-ADMMs using possibly indefinite proximal terms,
one may refer to the references \cite{JiWuCai20,XiCLi18}.


}

For convergence rate of ADMM, it is well-known that most of deterministic ADMM algorithms  \cite{BLXZ2018,CBSL20,FHLY15,GuHYa2014,HagerZhang2019,HLWY14,HeMaYuan2016,
JiWuCai20,SSY2017,XWu2011,YYan2015} enjoy a global
$\C{O}(1/T)$ ergodic convergence rate for convex separable optimization,
where $T$ is the iteration number.
Under the assumption that the subdifferential of each component
objective function is piecewise linear, Yang-Han \cite{YangHn16}
established linear convergence rate of ADMM for two-block
separable convex optimization.
Assuming that an error bound condition holds,   the dual stepsize
is sufficiently small and the coefficient matrices in the equality constraint have full column ranks,
Hong-Luo \cite{HongLuo17} showed a linear convergence rate of their multi-block ADMM.
Zhang et al. \cite{NZWZ20} developed a majorized ADMM with indefinite proximal  terms (iPADMM)
for a class of composite convex optimization problems, and  analyzed the convergence of this iPADMM  with
 a linear convergence rate under a local error bound condition.
Moreover, Chang et al. \cite{CBSL20} proposed a linearized symmetric  ADMM with indefinite proximal regularization and optimal proximal parameter for solving the multi-block separable convex optimization.
More recently, Yuan-Zeng-Zhang \cite{YZZ20} show that the local linear convergence of ADMM can
be guaranteed by a partial error bound condition. For more details about linear convergence rate under strongly convexity assumption, we refer the {interested} readers to
\cite{BaiHz2019,CaiHanYuan17,GoldsteinDSB14,LinMaZhang15,LiuShangCheng17} and the references therein.


\section{Preliminaries}
\subsection{Notations and assumptions}\label{Sepre}
Let  $\mathbb{R}$, $\mathbb{R}^n$, and $\mathbb{R}^{n\times l}$
be the sets of  real numbers,  $n$ dimensional real column vectors,
and $n\times l$  dimensional  real matrices, respectively.
The $\m{I}$ and $\m{0}$ denote the identity matrix and the zero matrix/vector, respectively.
For any  symmetric matrices $A$ and $B$ of the same dimension, $A \succ B$
($A \succeq B$) means $A - B$ is a positive definite (semidefinite) matrix.
For any symmetric matrix $G$, define $\|\m{x}\|_G^2 := \m{x} \tr G \m{x}$  and $\|\m{x}\|_G := \sqrt{\m{x} \tr G \m{x}}$ if
$G\succeq \m{0}$.
We use $\|\cdot\|$ to denote the standard Euclidean norm  equipped
with inner product $\langle \cdot,\cdot\rangle$,
$\nabla f(\m{x})$ to represent the gradient of $f$ at $\m{x}$, { and $\mathbb{E}[\cdot]$ to denote mathematical expectation}.
We also define{
\begin{equation}\label{vector-w-v}
\m{w}=\left(\begin{array}{c}
 \m{x}\\\m{y}\\ \LAMBDA
\end{array}\right),~~
\C{J}(\m{w})=\left(\begin{array}{c}
  -A\tr \LAMBDA\\ -B\tr\LAMBDA\\  A\m{x}+B\m{y}-\m{b}
\end{array}\right),
\end{equation}
and
\begin{equation}\label{vector-w-v12}
\m{w}^k=\left(\begin{array}{c}
  \m{x}^k\\\m{y}^k\\ \LAMBDA^k
\end{array}\right),~~
\C{J}(\m{w}^k)=\left(\begin{array}{c}
 -A\tr \LAMBDA^k\\ -B\tr\LAMBDA^k\\ A\m{x}^k+B\m{y}^k-\m{b}
\end{array}\right).
\end{equation}
For convenience of analysis, we {simply denote}   $F(\m{w}) =  f(\m{x}) + g(\m{y}).$}

We make the following two assumptions:
\smallskip
\begin{assumption} \label{asum-1}
The primal-dual solution set   $\Omega^*$ of  problem (\ref{P}) is nonempty, and the problem
$
\min\limits_{\m{y}\in \C{Y}} \left\{g(\m{y})+\frac{1}{2}\m{y}\tr B\tr B \m{y} +\m{z}\tr \m{y}\right\}
$
has a minimizer for any $\m{z}\in \mathbb{R}^{n_2}$.
 \smallskip
\end{assumption}
\begin{assumption} \label{asum-2}
For any $\C{H} \succ \m{0}$, there exists a constant $\nu > 0$ such that the gradients $\nabla f_j$ satisfy the Lipschitz condition
\begin{equation}\label{g-lipschitz}
\| \nabla f_j (\m{x}_1) - \nabla f_j (\m{x}_2) \|_{\C{H}^{-1}} \le \nu \|\m{x}_1 - \m{x}_2\|_{\C{H}}
\end{equation}
for every $\m{x}_1, \m{x}_2 \in \C{X}$ and $j=1,2,\cdots,N$.
\end{assumption}

The first assumption is a basic assumption to ensure the solvability of the problem. Under  Assumption \ref{asum-2}, it holds that for every $\m{x}, \m{y} \in \C{X}$, we have
\[
f(\m{x}_1)\leq f(\m{x}_2)+\langle \nabla f(\m{x}_2), \m{x}_1 - \m{x}_2 \rangle +
\frac{\nu}{2}\| \m{x}_1-\m{x}_2\|_{\C{H}}^2.
\]

\subsection{Variational characterization {of (\ref{P})}}

Denote $\Omega=\C{X}\times\C{Y}\times \mathbb{R}^n$.
{It's well-known in convex optimization that any saddle-point of
the Lagrangian (\ref{Lagr})  corresponds to a primal-dual solution of problem (\ref{P}).}
A point $\m{w}^*:=(\m{x}^*;\m{y}^*;\LAMBDA^*)\in \Omega$
is  called a saddle-point of $\C{L}\left(\m{x}, \m{y},\LAMBDA\right)$ if it satisfies
\begin{equation}\label{sadd-point}
\C{L}\left(\m{x}^*,\m{y}^*,\LAMBDA\right)
\leq\C{L}\left(\m{x}^*,\m{y}^*,\LAMBDA^*\right)
\leq \C{L}\left(\m{x},\m{y},\LAMBDA^*\right),
\quad \forall\m{w}\in \Omega,
\end{equation}
which  is equivalent to
\[
\left \{\begin{array}{lllll}
f(\m{x})- f(\m{x}^*) &+&
(\m{x}-\m{x}^*)\tr \left(-A\tr \LAMBDA^*\right)&\geq&  0, \\
 g(\m{y})- g(\m{y}^*) &+&
(\m{y}-\m{y}^*)\tr \left(- B\tr\LAMBDA^*\right)&\geq& 0, \\
&&A\m{x}^*+B\m{y}^*- \m{b} &=&  \m{0}.
\end{array}\right.
\]
Rewriting these inequalities as a more compact form, it  gives
\begin{equation}\label{Sec3-1}
\quad F(\m{w})- F(\m{w}^*) +(\m{w} -\m{w}^*)\tr \C{J}(\m{w}^*) \geq  0,\quad \forall \m{w}\in \Omega.
\end{equation}
Notice that  the affine  mapping $\C{J}(\cdot)$ is skew-symmetric. So, we have
\begin{eqnarray} \label{Sec1-J}
\left(\m{w}-\bar{\m{w}}\right)\tr
\left[\C{J}(\m{w})-\C{J}(\bar{\m{w}})\right]\equiv0,\quad \forall \m{w},\bar{\m{w}}\in \Omega.
\end{eqnarray}
Hence, (\ref{Sec3-1}) is also equivalent to
\begin{equation}\label{Sec3-1-jb}
\quad F(\m{w})- F(\m{w}^*) +(\m{w} -\m{w}^*)\tr \C{J}(\m{w}) \geq  0,\quad \forall\m{w}\in \Omega.
\end{equation}
{The above discussion shows  that the saddle-point  $\m{w}^*$ can be
also  characterized by the variational inequality (\ref{Sec3-1-jb}).}

\section{The proposed algorithm}\label{algor-main}
Motivated from the stochastic  AS-ADMM developed in \cite{BaiHz2019} and the deterministic GS-ADMM
proposed in \cite{BLXZ2018},  we now propose a Symmetric Accelerated Stochastic ADMM  (SAS-ADMM, {i.e., Algorithm~\ref{algo1}}),
{ which has the similarly dual stepsize   region $\Delta$ to that
of GS-ADMM defined as
\begin{equation}  \label{domainK}
\Delta=\left\{(\tau,s)\mid ~\tau+s>0,~\tau\leq 1,~ -\tau^2 -s^2 -\tau s +\tau +s +1\geq0\right\}.
\end{equation}
}
The main features of SAS-ADMM are summarized as follows:
\begin{itemize}
\item[(i)]
SAS-ADMM has many   {analogous advantages} to AS-ADMM developed in \cite{BaiHz2019}.
Specifically, SAS-ADMM has low memory requirement since there is no need to save
previous stochastic gradients and iterates. The subroutine $\mathbf{xsub}$ is a variant of deterministic accelerated gradient method
where the full gradient is replaced by a stochastic gradient.
{In addition, the users have the flexibility of choosing a zero mean random vector $\m{e}_t$ to reduce the variance of $\hat{\m{g}}_t$. A simple choice is $\m{e}_t=\m{0}$, while faster convergence is
observed in the numerical experiments when a variance reduction technique
is employed (see (\ref{et-set}) in Section~\ref{NumExp}). }
Under our blank assumption that $g$ is a proper convex function, the proximal {$\m{y}$}-subproblem is solvable.
Also, under the assumption that the projection onto the constraint set $\mathcal{X}$ is simple,
the iterations in subroutine $\mathbf{xsub}$
can be performed efficiently when both  $\C{M}_k$ and $\C{H}$ are multiples of identity matrix\footnote{As explained in Remark \ref{41re} and experiments, both $\C{M}_k$ and $\C{H}$
 could be chosen as multiples of identity matrix.
 So the $\breve{\m{x}}_{t+1}$-subproblem is equivalent to a projection onto $\C{X}$.}.
\item[(ii)]
Unlike the classical  ADMM and AS-ADMM \cite{BaiHz2019}, the dual variable of SAS-ADMM is
symmetrically updated twice { and allowed to use the large} stepsize region (\ref{domainK}).
{   SAS-ADMM will reduce to the  aforementioned PRSM  if the $\m{x}$-subproblem is solved
deterministically as in S-ADMM, $L=\m{0}$ and $(\tau,s)=(1,1)\in\Delta$.
When the stepsize $\tau=0$ and $L = \m{0}$, SAS-ADMM reduces to AS-ADMM with stepsize $s\in(0, \frac{1+\sqrt{5}}{2}]$(the half-open interval).}  Compared with the standard stepsize region {$(0, \frac{1+\sqrt{5}}{2})$} for the dual variable of ADMM, this symmetric updates of dual variable is more balanced, flexible
and often lead better numerical performance.

\renewcommand\figurename{Alg.}
\begin{figure}
\begin{center}
{ \tt
\begin{tabular}{l}
\hline
\\
{\bf Parameters:}
$\beta>0,~ \C{H} \succ \m{0},~ L\succeq \m{0}$ and $(\tau,s)\in \Delta$ given by (\ref{domainK}).\\
{\bf Initialization:}
$(\m{x}^0,\m{y}^0,\LAMBDA^0)$
$\in\C{X}\times\C{Y}\times \mathbb{R}^n,~ {\breve{\m{x}}^0=\m{x}^0}$.\\
For  $k=0,1,\ldots$ \\
\hspace*{.3in}Choose $m_k >0,~ \eta_k>0$ and $\C{M}_k$ such that
$\C{M}_k -  \beta A \tr A \succeq\m{0}$.\\
\hspace*{.3in}$\m{h}^k :=$
$-A\tr \left[\LAMBDA^k-\beta(A\m{x}^k+B\m{y}^k-\m{b})\right]$. \\
\hspace*{.3in}$(\m{x}^{k+1}, \breve{\m{x}}^{k+1}) =$
{\bf xsub} ($\m{x}^k, \breve{\m{x}}^k, \m{h}^k)$. \\
\hspace*{.3in}$\LAMBDA^{k+\frac{1}{2}}=\LAMBDA^k-
\tau \beta\left(A\m{x}^{k+1}+B\m{y}^{k}-\m{b}\right)$.\\
\hspace*{.3in}$\m{y}^{k+1}\in\arg\min\limits_{\m{y} \in \C{Y} }    \C{L}_\beta\left( \m{x}^{k+1}, \m{y},\LAMBDA^{k+\frac{1}{2}}\right)
 +\frac{1}{2}\left\|\m{y}-\m{y}^k\right\|_{L}^2  $.\\
\hspace*{.3in}$\LAMBDA^{k+1}=\LAMBDA^{k+\frac{1}{2}}-
s \beta\left(A\m{x}^{k+1}+B\m{y}^{k+1}-\m{b}\right).$\\
end\\
\hline
\\
$(\m{x}^+, \breve{\m{x}}^+) =$
{\bf xsub} ($\m{x}_1$, $\breve{\m{x}}_1$, $\m{h}$).\\
\textbf{For} $t=1$, 2, $\ldots$, $m_k$ \\
\hspace*{.3in}Randomly select $\xi_t \in \{1, \; 2,\;  \ldots,\;  N\}$
with uniform probability.\\
\hspace*{.3in}$\beta_t = 2/(t+1)$, $\gamma_t = 2/(t\eta_k)$,
$\hat{\m{x}}_t = \beta_t \breve{\m{x}}_t + (1-\beta_t)\m{x}_t$. \\
\hspace*{.3in}$\m{d}_t = \hat{\m{g}}_t + \m{e}_t$, where $\hat{\m{g}}_t = \nabla f_{\xi_t} (\hat{\m{x}}_t)$ and $\m{e}_t$ is a random vector\\
 \hspace*{1.2in}satisfying $\mathbb{E}[\m{e}_t]=\m{0}$. \\
\hspace*{.3in}$\breve{\m{x}}_{t+1}=\arg\min
\left\{\langle \m{d}_t + \m{h},\m{x}\rangle+\frac{\gamma_t}{2}
\left\|\m{x}-\breve{\m{x}}_{t}\right\|_{ \C{H} }^2
+\frac{1}{2} \left\|\m{x}-\m{x}^k \right\|_{ \C{M}_k}^2 :
\m{x} \in \C{X} \right\}.$ \\
\hspace*{.3in}$\m{x}_{t+1}=$
$\beta_t \breve{\m{x}}_{t+1}+(1-\beta_t)\m{x}_{t}$.\\
\textbf{end}\\
\textbf{Return} $(\m{x}^+, \breve{\m{x}}^+) =
(\m{x}_{m_k+1}, \breve{\m{x}}_{m_k+1})$.\\
\hline
\end{tabular}
}
\end{center}
\caption{ {Symmetric} accelerated stochastic ADMM (SAS-ADMM) with lager stepsizes}\label{algo1}
\end{figure}
\renewcommand\figurename{Fig.}

\item[(iii)]
If $m_k=1, N=1$, {then} SAS-ADMM degrades to a linearized symmetric ADMM.
When $m_k>1, N=1$, SAS-ADMM is a multi-step deterministic inexact symmetric ADMM.
Hence, the convergence properties developed in this
paper also apply to these deterministic algorithms as special cases.
Moreover, by taking  $L=\gamma \m{I}-\beta B\tr B$ for some $\gamma >0$, the $\m{y}$-subproblem
 would become the following proximal mapping problem:
\begin{equation}\label{prox-y-problem}
\textrm{prox}_\gamma^g(\m{y}_c^k):=\arg\min\limits_{\m{y} \in \C{Y} }    g(\m{y})+ \frac{\gamma}{2}
\left\|   \m{y} -\m{y}_c^k
\right\|^2,
\end{equation}
where $\m{y}_c^k={\m{y}^k -} {\beta B\tr (A\m{x}^{k+1}+B\m{y}^k-\m{b}- {\LAMBDA^{k+\frac{1}{2}}}/{\beta})}/{\gamma}$.
In this case, {the Assumption \ref{asum-1} is not required since strong convexity
of the $\m{y}$-subproblem implies a unique global solution and  a}  closed-form solution
may exist when $g$ has certain structure.
\item[(iv)]
{
With the aid of variational analysis, we show that SAS-ADMM has the worst-case $\C{O}(1/T)$ ergodic
convergence rate in terms of the expectation of both the objective value gap and the constraint violation,
where $T$ is the number of the outer iterations. Preliminary  experiments and results show that SAS-ADMM performs competitively well and often {slightly} better than AS-ADMM \cite{BaiHz2019}
for solving a family of  separable convex optimization problems arising from big-data applications. }

\end{itemize}

\section{Convergence analysis}\label{Theor res}
{To establish   convergence of Algorithm~\ref{algo1}, we first need the following lemma about the iterates generated by the {\bf xsub} routine in
Algorithm~\ref{algo1}. The lemma was given in  \cite{BaiHz2019} and thus  we omit its proof.}
\begin{lemma} \cite[Lemma 3.2]{BaiHz2019} \label{Sec31-3}
Let   $\DELTA_t=\nabla f(\hat{\m{x}}_{t})-\m{d}_t$.
Suppose $\eta_k\in(0,  {1}/{\nu})$ { and Assumption \ref{asum-2} holds}. Then,
the iterates generated by Algorithm~\ref{algo1} satisfy
\begin{equation}\label{Sec3-5}
\qquad f(\m{x})-f(\m{x}^{k+1}) + \left\langle \m{x} -
\m{x}^{k+1},-A\tr\tilde{\LAMBDA}^k\right\rangle \geq
\left\langle \m{x}^{k+1}-\m{x} , \C{D}_k (\m{x}^{k+1}
-\m{x}^k) \right\rangle+\zeta^k
\end{equation}
for all $\m{x} \in \C{X}$, where
\begin{equation}\label{Sec3-5-jb}
\tilde{\LAMBDA}^k=\LAMBDA^k-\beta \left( A \m{x}^{k+1}
+B \m{y}^{k}  - \m{b} \right),~   {\C{D}_k = \C{M}_k - \beta A \tr A}  \quad \textrm{and}~
\end{equation}
\begin{eqnarray}
&&\zeta^k = \frac{2}{m_k(m_k+1)}
\bigg[\frac{1}{\eta_k}\left(\left\|\m{x}
-\breve{\m{x}}^{k+1}\right\|_{\C{H}}^2-
\left\|\m{x}-\breve{\m{x}}^{k}\right\|_{\C{H}}^2\right)
\label{zeta_k} \\
&& ~~~  ~-\sum\limits_{t=1}^{m_k}t\langle \DELTA_t, \breve{\m{x}}_{t}
-\m{x}\rangle -\frac{\eta_k}{4(1-\eta_k\nu)}
\sum\limits_{t=1}^{m_k}t^2\left\|\DELTA_t\right\|_{\C{H}^{-1}}^2\bigg].
\nonumber
\end{eqnarray}
\end{lemma}

Based on the above lemma, we {can immediately establish the following result}.
\begin{lemma} \label{Sec31-bz6}
Suppose $\eta_k\in(0,  {1}/{\nu})$. {Then, the} iterates generated by Algorithm \ref{algo1} satisfy
\begin{equation} \label{Sec3-bz7}
\quad F(\m{w})-F(\tilde{\m{w}}^k) +
\left\langle \m{w} - \tilde{\m{w}}^k,
\C{J}(\m{w})\right\rangle
\geq  ( \m{w}-\tilde{\m{w}}^k)\tr
Q_k (\m{w}^k-\tilde{\m{w}}^k) + \zeta^k
\end{equation}
for all $\m{w} \in \Omega$,
where   $\zeta^k$  and {$\tilde{\g{\lambda}}^k$ are defined in  (\ref{zeta_k}) and (\ref{Sec3-5-jb})},
\begin{equation}\label{Sec31-vwuJ}
\qquad\tilde{\m{w}}^k =
\left(\begin{array}{c}
  \m{\tilde{x}}^k\\ \m{\tilde{y}}^k\\  \tilde{\LAMBDA}^k
\end{array}\right) =
\left(\begin{array}{c}
  \m{x}^{k+1}\\ \m{y}^{k+1}\\    \tilde{\LAMBDA}^k
\end{array}\right)~ \mbox{ and }~
Q_k=\left[\begin{array}{ccc}
           \C{D}_k &&   \\
          &L+\beta B\tr B&  -\tau B\tr  \\
      & -B &     \frac{1}{\beta} \m{I}
\end{array}\right].
\end{equation}
\end{lemma}
\begin{proof}
By the first-order optimality condition of the $\m{y}$-subproblem, we have
\begin{equation}\label{Chapt5-a-Sec31-xN}
g(\m{y}) - g(\m{y}^{k+1})
+ \left\langle\m{y}-\m{y}^{k+1}, \m{p}_k \right\rangle\geq  0,
\quad \forall \m{y}\in \C{Y},
\end{equation}
where $\m{p}_k$ is the gradient of the smooth terms in the objective function of the $\m{y}$-subproblem:
\begin{eqnarray*}
\m{p}_k&=&
-B\tr  \LAMBDA^{k+\frac{1}{2}} +\beta B\tr\left(A\m{x}^{k+1}+B\m{y}^{k+1}-\m{b}\right)  +L(\m{y}^{k+1}-\m{y}^k)
\nonumber\\
&=& -B\tr  \LAMBDA^{k+\frac{1}{2}} +\beta B\tr \left(A\m{x}^{k+1}+B\m{y}^k-\m{b}\right)+  \left[
L
+ \beta B\tr B \right] \left(\m{y}^{k+1} - \m{y}^k  \right)  \nonumber \\
&=&  -B\tr  \LAMBDA^{k+\frac{1}{2}} +  B\tr \left( \LAMBDA^k- \tilde{\LAMBDA}^k\right)+  \left[
L
+ \beta B\tr B \right] \left(\m{y}^{k+1} - \m{y}^k  \right)\\
&=& -B\tr \tilde{\LAMBDA}^k + \tau B\tr( \LAMBDA^{k}-\tilde{\LAMBDA}^{k}) +
 \left[ L+ \beta B\tr B \right] \left(\m{y}^{k+1} - \m{y}^k  \right).
\end{eqnarray*}
The above last equality uses the following relation
\begin{equation}\label{LA12}
\LAMBDA^{k+\frac{1}{2}}= \LAMBDA^{k}-\tau( \LAMBDA^{k}-\tilde{\LAMBDA}^{k}).
\end{equation}
By the definition of $\tilde{\LAMBDA}^k$, we have
\begin{equation} \label{bjc-100}
\left( A \tilde{\m{x}}^k +B \tilde{\m{y}}^k
- \m{b} \right) - B \left(  \tilde{\m{y}}^k -  \m{y}^k \right) +\frac{1}{\beta}\left(\tilde{\LAMBDA}^k - \LAMBDA^k\right)={\m{0}}.
\end{equation}
Taking inner product of the above equality with ${\LAMBDA} -\tilde{\LAMBDA}^k$, we get
\begin{equation} \label{bjc-100zj}
\qquad~~ \left\langle  {\LAMBDA} -\tilde{\LAMBDA}^k, A \tilde{\m{x}}^k +B \tilde{\m{y}}^k
- \m{b} \right\rangle= \left\langle {\LAMBDA} -\tilde{\LAMBDA}^k,
 - B \left(    \m{y}^k-\tilde{\m{y}}^k  \right) +\frac{1}{\beta}\left( \LAMBDA^k- \tilde{\LAMBDA}^k \right)\right\rangle.
\end{equation}
Then, the inequality (\ref{Sec3-bz7}) is achieved by combining
 (\ref{Sec3-5}), (\ref{Chapt5-a-Sec31-xN}), (\ref{bjc-100zj}) together with the property (\ref{Sec1-J}).
\end{proof}
\smallskip

\subsection{More technical results} \label{sec322-key}
We show the following corollaries for establishing the main convergence theorem
of Algorithm~\ref{algo1}.

\smallskip
\begin{corollary}\label{coll37}
Suppose $\eta_k\in(0,  1/\nu)$. {Then,
the} iterates generated by Algorithm \ref{algo1} satisfy
\begin{eqnarray} \label{bjc-39}
&&F(\m{w})-F(\tilde{\m{w}}^k) +
( \m{w} - \tilde{\m{w}}^k)\tr \C{J}(\m{w})\\
& &\ge
\frac{1}{2}\left\{\left\|\m{w}-\m{w}^{k+1}\right\|^2_{\tilde{Q}_k}
-\left\|\m{w}-\m{w}^{k}\right\|^2_{\tilde{Q}_k}
+  \left\|\m{w}^k-\tilde{\m{w}}^k\right\|_{\tilde{G}_k}^2 \right\} + \zeta^k \nonumber
\end{eqnarray}
for all $\m{w} \in \Omega$, where {$\zeta^k$  is defined in {(\ref{zeta_k})} and}
\begin{equation}\label{tilde-Q}
 {\scriptsize\tilde{Q}_k=\left[\begin{array}{ccc}
           \C{D}_k &&   \\
          &L+\left(1-\frac{\tau s}{\tau+s}\right)\beta B\tr B&  -\frac{\tau}{\tau+ s} B \tr  \\
      & -\frac{\tau}{\tau+ s} B &    \frac{1}{\beta(\tau +s)} \m{I}
\end{array}\right], ~
\tilde{G}_k=\left[\begin{array}{ccc}
           \C{D}_k &&   \\
          &L+ (1-s)\beta B\tr B&  (s-1) B \tr  \\
      & (s-1) B &    \frac{2-  \tau-s}{\beta} \m{I}
\end{array}\right]}.
\end{equation}

\end{corollary}
\begin{proof}
By (\ref{LA12}) and the way of generating   $\LAMBDA^{k+1}$, we have
\begin{equation}\label{121-zlg}
- s \beta B \left( \m{y}^k-\tilde{\m{y}}^k \right) + (\tau+s) \left(\LAMBDA^k-\tilde{\LAMBDA}^k \right) =
\LAMBDA^k- \LAMBDA^{k+1},
\end{equation}
which, by the definition of $\tilde{\m{w}}^k$ in (\ref{Sec31-vwuJ}), further shows
\begin{equation}\label{def-P}
\quad\m{w}^k - \m{w}^{k+1} = P \left(  \m{w}^k-\tilde{\m{w}}^k  \right)~~ \textrm{with}~~
P=\left[\begin{array}{ccccc}
           \m{I} && &&  \\
          && \m{I} &&    \\
      && -s \beta B  & &   (\tau+s) \m{I}
\end{array}\right].
\end{equation}
Hence,   the   relation
$
Q_k ( \m{w}^k-\tilde{\m{w}}^k ) =Q_kP^{-1}(\m{w}^k - \m{w}^{k+1})
$
holds and
\[
Q_kP^{-1}=\left[\begin{array}{ccc}
           \C{D}_k &&   \\
          &L+\left(1-\frac{\tau s}{\tau+s}\right)\beta B\tr B&  -\frac{\tau}{\tau+ s} B \tr  \\
      & -\frac{\tau}{\tau+ s} B &    \frac{1}{\beta(\tau +s)} \m{I}
\end{array}\right]=\tilde{Q}_k.
\]
For any $\m{w}\in \Omega$, it follows from   (\ref{Sec3-bz7}) and the above relation  that
\begin{eqnarray}\label{Sec3-bjc-1}
 &&F(\m{w})-F(\tilde{\m{w}}^k) +
( \m{w} - \tilde{\m{w}}^k)\tr \C{J}(\m{w})  \\
&& \ge    \zeta^k+ ( \m{w}- \tilde{\m{w}}^k)\tr \tilde{Q}_k (\m{w}^k - \m{w}^{k+1}) =   \zeta^k+\frac{1}{2}\left\{\left\|\m{w}-\m{w}^{k+1}\right\|^2_{\tilde{Q}_k}\right.\nonumber\\
&&~\left.-
\left\|\m{w}-\m{w}^{k}\right\|^2_{\tilde{Q}_k}+ \left\|\m{w}^k-\tilde{\m{w}}^k\right\|_{\tilde{Q}_k}^2 -
\left\|\m{w}^{k+1}-\tilde{\m{w}}^k\right\|_{\tilde{Q}_k}^2\right\}, \nonumber
 \end{eqnarray}
where   the   equality uses the   identity
\[
 2(a-b)\tr  \tilde{Q}_k(c-d)=   \|a-d\|_{\tilde{Q}_k}^2-\|a-c\|_{\tilde{Q}_k}^2 + \|c-b\|_{\tilde{Q}_k}^2-\|b-d\|_{\tilde{Q}_k}^2
\]
with specifications $ a:=\m{w}, b:= \tilde{\m{w}}^k,  c:=\m{w}^k,   d:=\m{w}^{k+1}.$

Now, by (\ref{def-P}) again, we deduce
\begin{eqnarray*}
&& \left\|\m{w}^k-\tilde{\m{w}}^k\right\|_{\tilde{Q}_k}^2 -  \left\|\m{w}^{k+1}-\tilde{\m{w}}^k\right\|_{\tilde{Q}_k}^2 \\
& =&  \left\|\m{w}^k-\tilde{\m{w}}^k\right\|_{\tilde{Q}_k}^2 -  \left\|\m{w}^{k+1}- \m{w}^k + \m{w}^k - \tilde{\m{w}}^k\right\|_{\tilde{Q}_k}^2 \\
& =& \left\|\m{w}^k-\tilde{\m{w}}^k\right\|_{\tilde{Q}_k}^2 - \left\| \m{w}^k - \tilde{\m{w}}^k - P ( \m{w}^k-\tilde{\m{w}}^k )\right\|_{\tilde{Q}_k}^2 \\
& = & \left\|\m{w}^k-\tilde{\m{w}}^k\right\|_{\tilde{G}_k}^2,
\end{eqnarray*}
where we uses the relation
$ \tilde{G}_k  =  P \tr \tilde{Q}_k + \tilde{Q}_k P - P \tr \tilde{Q}_k P
$ to obtain the last equality. Then, (\ref{bjc-39}) follows from (\ref{Sec3-bjc-1}).
\end{proof}
\smallskip

{
In the above Corollary \ref{coll37} and its proof, since  $\tilde{Q}_k$ is not necessarily positive semidefinite for any parameter $\tau$,
we abuses  the notation $\|\m{w}^k\|_{\tilde{Q}_k}^2 :=  (\m{w}^k) \tr \tilde{Q}_k \m{w}^k$.
Next, we provide a sufficient condition {to ensure}  the positive semidefiniteness of $\tilde{Q}_k$.
}

\begin{lemma} \label{Posity-1}
Let $L \succeq (\tau - 1) \beta B \tr B$. Then, the matrix $\tilde{Q}_k$ given by   (\ref{tilde-Q})  is symmetric positive semidefinite for any $(\tau,s)\in \Delta$.
\end{lemma}
\begin{proof}
Clearly, we just need to check the lower-upper 2-by-2 block of $\tilde{Q}_k$, i.e.,
\begin{eqnarray*}
&&\tilde{Q}_k^{L}=\left[\begin{array}{cc}
          L+\left(1-\frac{\tau s}{\tau+s}\right)\beta B\tr B&  -\frac{\tau}{\tau+ s} B \tr  \\
       -\frac{\tau}{\tau+ s} B &    \frac{1}{\beta(\tau +s)} \m{I}
\end{array}\right]\\
&&\succeq\left[\begin{array}{cc}
          \left(\tau-\frac{\tau s}{\tau+s}\right)\beta B\tr B&  -\frac{\tau}{\tau+ s} B \tr  \\
       -\frac{\tau}{\tau+ s} B &    \frac{1}{\beta(\tau +s)} \m{I}
\end{array}\right]\\
&&=\left[\begin{array}{cc}
          \beta^{\frac{1}{2}} B&    \\
        &    \beta^{-\frac{1}{2}}  \m{I}
\end{array}\right]\tr \left[\begin{array}{cc}
         \frac{\tau^2}{\tau+s} \m{I}&  -\frac{\tau}{\tau+ s} \m{I}  \\
       -\frac{\tau}{\tau+ s} \m{I} &    \frac{1}{\tau +s} \m{I}
\end{array}\right] \left[\begin{array}{cc}
          \beta^{\frac{1}{2}} B&    \\
        &    \beta^{-\frac{1}{2}}  \m{I}
\end{array}\right]
\end{eqnarray*}
is positive semidefinite. Notice that
\begin{eqnarray*}
\left[\begin{array}{cc}
          \m{I}&    \\
       \tau \m{I} &    \m{I}
\end{array}\right]\tr \left[\begin{array}{cc}
          \frac{\tau^2}{\tau+s} \m{I}&  -\frac{\tau}{\tau+ s} \m{I}  \\
       -\frac{\tau}{\tau+ s} \m{I} &    \frac{1}{\tau +s} \m{I}
\end{array}\right] \left[\begin{array}{cc}
          \m{I}&    \\
       \tau \m{I} &    \m{I}
\end{array}\right]
=\left[\begin{array}{cc}
            \m{0}&     \\
        &    \frac{1}{\tau +s} \m{I}
\end{array}\right].
\end{eqnarray*}
So, $\tilde{Q}_k^{L}$ is positive semidefinite since
$\tau+s>0$ for any $(\tau,s)\in \Delta$.
\end{proof}

\medskip

  {To show} the global convergence of Algorithm \ref{algo1},
we need to further establish a useful lower bound on the term
$ \left\|\m{w}^k-\tilde{\m{w}}^k\right\|_{\tilde{G}_k}^2$, since $\tilde{G}_k$ is not necessarily
positive definite for any $(\tau,s)\in\Delta$.
In the {following lemma,}   we assume  $L \succeq \m{0}$, which
implies  $L \succeq (\tau - 1) \beta B \tr B$ since $\tau \le 1$, and as a consequence,
Lemma~\ref{Posity-1} holds.

\smallskip

\begin{lemma} \label{key-lemma}
Let $L \succeq \m{0}$. Then, for any $(\tau,s)\in \Delta$ defined in (\ref{domainK}),
we have
\begin{eqnarray}\label{key-lemma-j12b}
&&  \left\|\m{w}^k-\tilde{\m{w}}^k\right\|_{\tilde{G}_k}^2  \geq
\left\| \m{x}^k - \m{x}^{k+1}\right\|_{\C{D}_k}^2
 +\omega_0 \left\| A\m{x}^{k+1}+B\m{y}^{k+1}-\m{b} \right\|^2   \\
& &
+ \omega_1 \left(\left\| A\m{x}^{k+1}+B\m{y}^{k+1}-\m{b} \right\|^2 -\left\| A\m{x}^k+B\m{y}^k-\m{b} \right\|^2 \right)\nonumber\\
&&+ \omega_2\left( \left\|\m{y}^k-\m{y}^{k+1}\right\|_L^2 -
 \left\|\m{y}^{k-1}-\m{y}^k\right\|_L^2 \right),\nonumber
\end{eqnarray}
where $\omega_i \ge 0$, $i=0, 1, 2$, are given as
\begin{equation}\label{def_omega}
\omega_0=\left(2-\tau-s -\frac{(1-s)^2}{1+\tau}\right)\beta, ~~
\omega_1=\frac{(1-s)^2}{1+\tau}\beta \; \mbox{ and } \;  \omega_2= \frac{1-\tau}{ 1+\tau  }.
\end{equation}
\end{lemma}
\begin{proof}
For any $(\tau,s) \in \Omega$, we have $1+\tau >0$.
By the structure of $\tilde{G}_k$ in (\ref{tilde-Q}), $L \succeq \m{0}$ and (\ref{bjc-100}), we have
\begin{eqnarray}\label{key-lemma-jb1}
&& \left\|\m{w}^k-\tilde{\m{w}}^k\right\|_{\tilde{G}_k}^2
  =  \left\|\m{x}^k - \m{x}^{k+1}\right\|_{\C{D}_k}^2
+ \left\| \m{y}^k- \m{y}^{k+1}  \right\|^2_{L+(1-s)\beta B\tr B}   \\
&&~~ + 2(s-1) \left( \LAMBDA^k-\tilde{\LAMBDA}^k \right) \tr B \left( \m{y}^k- \m{y}^{k+1} \right)
+ \frac{2-\tau-s}{\beta}  \left\| \LAMBDA^k-\tilde{\LAMBDA}^k \right\|^2 \nonumber\\
&& \geq  \left\| \m{x}^k - \m{x}^{k+1}\right\|_{\C{D}_k}^2
 + (2-\tau-s)\beta \left\| A\m{x}^{k+1}+B\m{y}^{k+1}-\m{b} \right\|^2
\nonumber \\
&&~~  + (1-\tau)\beta \left\| B \left( \m{y}^k- \m{y}^{k+1} \right) \right\|^2 \nonumber \\
&&~~  + 2(1-\tau)\beta \left( A\m{x}^{k+1}+B\m{y}^{k+1}- b \right) \tr B \left( \m{y}^k- \m{y}^{k+1} \right).\nonumber
\end{eqnarray}

We now estimate the last crossing term in (\ref{key-lemma-jb1}).
{Taking}  $\m{y} = \m{y}^k$ in the first-order optimality condition (\ref{Chapt5-a-Sec31-xN})   {yields}
\[
{\small g(\m{y}^k)- g(\m{y}^{k+1})
 + \left\langle  \m{y}^k-\m{y}^{k+1},-B\tr  \LAMBDA^{k+\frac{1}{2}} +\beta B\tr (A\m{x}^{k+1}+B\m{y}^{k+1}-\m{b} )  +L(\m{y}^{k+1}-\m{y}^k) \right\rangle\geq  0}.
\]
Similarly,   letting $\m{y} = \m{y}^{k+1}$ in the first-order optimality condition of the $\m{y}$-subproblem at the $(k-1)$-th iteration gives
\[
{\small g(\m{y}^{k+1}) - g(\m{y}^k)
 + \left\langle  \m{y}^{k+1}-\m{y}^k,-B\tr  \LAMBDA^{k-\frac{1}{2}} +\beta B\tr\left(A\m{x}^{k}+B\m{y}^{k}-\m{b}\right)  +L(\m{y}^{k}-\m{y}^{k-1}) \right\rangle\geq  0}.
\]
Summing up the above two inequalities together with the relation
 \[
 \LAMBDA^{k-\frac{1}{2}}-\LAMBDA^{k+\frac{1}{2}} =  \tau \beta \left(A\m{x}^{k+1}+B\m{y}^{k+1}-\m{b}\right)
 + s\beta\left(A\m{x}^k+B\m{y}^k-\m{b}\right)+\tau\beta B(\m{y}^k-\m{y}^{k+1}),
 \]
and noticing that $1+\tau >0$,  we obtain
\begin{eqnarray}\label{key-lemma-jb2}
& &   \left( A\m{x}^{k+1}+B\m{y}^{k+1}- \m{b} \right) \tr B \left( \m{y}^k- \m{y}^{k+1} \right)   \\
& \ge &  \frac{1-s}{1+\tau} \left( A\m{x}^k+B\m{y}^k- \m{b} \right) \tr B \left( \m{y}^k- \m{y}^{k+1} \right) -\frac{\tau}{1+\tau} \left\|B(\m{y}^k-\m{y}^{k+1}) \right\|^2\nonumber\\
&&+\frac{1}{\beta(1+\tau)}\left( \m{y}^k- \m{y}^{k+1} \right)\tr L\left[ \left(\m{y}^k-\m{y}^{k+1} \right)
- \left(\m{y}^{k-1}-\m{y}^k\right)\right]\nonumber \\
& \ge &  \frac{1-s}{1+\tau} \left( A\m{x}^k+B\m{y}^k- \m{b} \right) \tr B \left( \m{y}^k- \m{y}^{k+1} \right) -\frac{\tau}{1+\tau} \left\|B(\m{y}^k-\m{y}^{k+1}) \right\|^2\nonumber\\
&&+\frac{1}{2 \beta(1+\tau)}\left( \left\| \m{y}^k- \m{y}^{k+1}\right\|_L^2 -
\left\| \m{y}^{k-1}- \m{y}^k\right\|_L^2\right).\nonumber
\end{eqnarray}
Then, combining  {(\ref{key-lemma-jb1}) and (\ref{key-lemma-jb2})}  we have
\begin{eqnarray*}
 &&\left\|\m{w}^k-\tilde{\m{w}}^k\right\|_{\tilde{G}_k}^2\\
  &\geq& \left\| \m{x}^k - \m{x}^{k+1}\right\|_{\C{D}_k}^2 +
   (2-\tau-s)\beta \left\| A\m{x}^{k+1}+B\m{y}^{k+1}-\m{b} \right\|^2\\
& &
+ \frac{2\beta}{1+\tau} (1-\tau)(1-s)\left( A\m{x}^k+B\m{y}^k- \m{b} \right) \tr B \left( \m{y}^k- \m{y}^{k+1} \right)\\
&& + (1-\tau)\beta   \left\| B \left( \m{y}^k- \m{y}^{k+1} \right) \right\|^2 -\frac{2\tau(1-\tau)}{1+\tau}\beta  \left\| B \left( \m{y}^k- \m{y}^{k+1} \right) \right\|^2\\
&&+  \frac{1-\tau}{1+\tau} \left( \left\|\m{y}^k-\m{y}^{k+1}\right\|_L^2 - \left\|\m{y}^{k-1}-\m{y}^k\right\|_L^2 \right) \\
&\geq&\left\| \m{x}^k - \m{x}^{k+1}\right\|_{\C{D}_k}^2  +
   \left(2-\tau-s -\frac{(1-s)^2}{1+\tau}\right)\beta \left\| A\m{x}^{k+1}+B\m{y}^{k+1}-\m{b} \right\|^2\\
& &
+ \frac{(1-s)^2}{1+\tau}\beta \left(\left\| A\m{x}^{k+1}+B\m{y}^{k+1}-\m{b} \right\|^2 -\left\| A\m{x}^k+B\m{y}^k-\m{b} \right\|^2 \right)\\
&&+  \frac{1-\tau}{ 1+\tau } \left( \left\|\m{y}^k-\m{y}^{k+1}\right\|_L^2 -   \left\|\m{y}^{k-1}-\m{y}^k\right\|_L^2 \right) \\
&&+ \left(1 - \tau - \frac{(1-\tau)^2}{1+\tau} - \frac{2 \tau (1-\tau)}{1+\tau} \right)\beta\left\|B(\m{y}^k-\m{y}^{k+1})\right\|^2 \\
&=&\left\| \m{x}^k - \m{x}^{k+1}\right\|_{\C{D}_k}^2  +
   \left(2-\tau-s -\frac{(1-s)^2}{1+\tau}\right)\beta \left\| A\m{x}^{k+1}+B\m{y}^{k+1}-\m{b} \right\|^2\\
& &
+ \frac{(1-s)^2}{1+\tau}\beta \left(\left\| A\m{x}^{k+1}+B\m{y}^{k+1}-\m{b} \right\|^2 -\left\| A\m{x}^k+B\m{y}^k-\m{b} \right\|^2 \right)\\
&&+  \frac{1-\tau}{ 1+\tau } \left( \left\|\m{y}^k-\m{y}^{k+1}\right\|_L^2 -   \left\|\m{y}^{k-1}-\m{y}^k\right\|_L^2 \right),
\end{eqnarray*}
where the second inequality follows from the Cauchy-Schwartz inequality
\begin{eqnarray*}
&&2(1-s)(1-\tau) \left( A\m{x}^k+B\m{y}^k- \m{b} \right) \tr B \left( \m{y}^k- \m{y}^{k+1} \right)\\
&\geq&-(1-s)^2\left\|A\m{x}^k+B\m{y}^k- \m{b}\right\|^2 -(1-\tau)^2\left\| B \left( \m{y}^k- \m{y}^{k+1} \right)\right\|^2.
\end{eqnarray*}
So,   {(\ref{key-lemma-j12b}) holds} with  $\omega_i$, $i=0, 1, 2$, defined as in (\ref{def_omega}).
Moreover, for any {$(\tau,s) \in \Delta$,} we can derive $\omega_i \ge 0$ for $i= 0, 1, 2$.
This completes the whole proof.
\end{proof}

\subsection{Iteration complexity in expectation} \label{Sec-var1}
We now analyze the global ergodic convergence and the iteration complexity of Algorithm~\ref{algo1}.
\smallskip

\begin{theorem} \label{Sec3-theore1}
Suppose $L \succeq \m{0}$ and $(\tau,s)\in \Delta$ defined in (\ref{domainK}).
If for some integers $\kappa, T>0$,
the following conditions hold for all  $k\in[\kappa, \kappa+T]$: (I)
$\eta_k\in (0, 1/(2\nu)]$ and
the sequence $\left\{ \eta_km_k(m_k+1)\right\}$ is nondecreasing;
(II) $ \C{D}_k \succeq \C{D}_{k+1} \succeq \m{0}$ and $\mathbb{E}(\|\DELTA_t\|_{\C{H}^{-1}}^2) \le \sigma^2$
for some $\sigma > 0$, where $\DELTA_t$ and $\C{D}_k$ are defined in Lemma~\ref{Sec31-3}. \\
Then, for  any $ \m{w} \in \Omega$, we have
  \begin{eqnarray}\label{Ex-F}
&& \mathbb{E}\left[ F(\m{w}_T)
-F(\m{w})+(\m{w}_T-\m{w})\tr \mathcal{J}(\m{w})\right]\\
&&\leq
\frac{1}{2T}\bigg\{ \sigma^2\sum\limits_{k=\kappa}^{\kappa+T} \eta_k m_k
+ \frac{4}{m_\kappa(m_\kappa+1)\eta_\kappa} \left\|\m{x}- \breve{\m{x}}^\kappa\right\|_{\C{H}}^2 + \left\|\m{w}-\m{w}^\kappa\right\|^2_{\tilde{Q}_\kappa} \nonumber\\
&& ~   + \omega_1 \left\|A\m{x}^\kappa+B\m{y}^\kappa-\m{b}\right\|^2+\omega_2\left\|\m{y}^{\kappa-1}-\m{y}^\kappa\right\|_L^2   \bigg\}, \nonumber
\end{eqnarray}
where $\m{w}_T =\frac{1}{T}\sum_{k=\kappa}^{\kappa+T}\tilde{\m{w}}^{k} $,
$ \omega_1\geq 0$ and $\omega_2\geq 0$ are defined in (\ref{def_omega}).
\end{theorem}
\begin{proof}
By the assumption, $\C{D}_k \succeq \C{D}_{k+1} \succeq \m{0}$
  implies {${\tilde{Q}_k} \succeq {\tilde{Q}_{k+1}} \succeq \m{0}$ for   the matrix $\tilde{Q}_k$ given  in   (\ref{tilde-Q}).}
Substituting (\ref{key-lemma-j12b}) into (\ref{bjc-39}) and utilizing the relation
${\tilde{Q}_k} \succeq {\tilde{Q}_{k+1}}$, it follows from {{Lemma}}~\ref{key-lemma} that
\begin{eqnarray*}
&& F(\tilde{\m{w}}^k) - F(\m{w}) +
( \tilde{\m{w}}^k - \m{w})\tr \mathcal{J}(\m{w}) \le
\label{bjc-39-new} \\
&-&\zeta^k + \frac{1}{2}\left\{\left\|\m{w}-\m{w}^{k}\right\|^2_{\tilde{Q}_k}
-\left\|\m{w}-\m{w}^{k+1}\right\|^2_{\tilde{Q}_{k+1}} \right\}   \\
&+& \frac{\omega_1}{2} \left(\left\| A\m{x}^k+B\m{y}^k-\m{b} \right\|^2 -\left\| A\m{x}^{k+1}+B\m{y}^{k+1}-\m{b} \right\|^2 \right) \\
&+&  \frac{\omega_2}{2}\left(    \left\|\m{y}^{k-1}-\m{y}^k\right\|_L^2 -\left\|\m{y}^k-\m{y}^{k+1}\right\|_L^2 \right),
\end{eqnarray*}
where {$ \omega_1, \omega_2$} are defined in (\ref{def_omega}).
Summing the above inequality  over $k$ between $\kappa$ and $\kappa+T$,
we   deduce by  Lemma \ref{Posity-1} that
\begin{eqnarray}\label{sumover}
&&  \sum\limits_{k=\kappa}^{\kappa+T}F(\tilde{\m{w}}^k)-T\left\{ F(\m{w})
+ \left(  \m{w}_T - \m{w} \right) \tr \mathcal{J}(\m{w})\right\} \leq -\sum\limits_{k=\kappa}^{\kappa+T}\zeta^k\\
&& + \frac{1}{2}
\left\{ \|\m{w}-\m{w}^\kappa\|^2_{\tilde{Q}_\kappa} +
\omega_1 \left\|A\m{x}^\kappa+B\m{y}^\kappa-\m{b}\right\|^2  +\omega_2\left\|\m{y}^{\kappa-1}-\m{y}^\kappa\right\|_L^2  \right\}.\nonumber
\end{eqnarray}
Then, it follows from convexity of $F$ and the definition of $\textbf{w}_T$ that
\begin{equation} \label{Sec3-bjc-6}
F(\m{w}_T)\leq \frac{1}{T}\sum_{k=\kappa}^{\kappa+T}F(\tilde{\m{w}}^{k}).
\end{equation}
Dividing (\ref{sumover}) by $T$ and using (\ref{Sec3-bjc-6}), we obtain
\begin{eqnarray} \label{Sec3-bjc-5}
&&  F(\m{w}_T) -F(\m{w})+ (\m{w}_T-\m{w})\tr \mathcal{J}(\m{w}) \le \frac{1}{T}\left[ -\sum\limits_{k=\kappa}^{\kappa+T}\zeta^k\right.\\
& & \left.+ \frac{1}{2}
\left\{ \|\m{w}-\m{w}^\kappa\|^2_{\tilde{Q}_\kappa}
+ \omega_1 \left\|A\m{x}^\kappa+B\m{y}^\kappa-\m{b}\right\|^2+\omega_2\left\|\m{y}^{\kappa-1}-\m{y}^\kappa\right\|_L^2 \right\} \right]. \nonumber
\end{eqnarray}

Let us now focus on the terms involving $\zeta^k$.
By assumption, the sequence $\left\{m_k(m_k+1)\eta_k\right\}$ is nondecreasing for   $k\in[\kappa, \kappa+T]$
and $\C{H} \succ \m{0}$, thus  we have
\begin{equation}\label{Sec3-bcz-6}
\qquad~~  \sum\limits_{k=\kappa}^{\kappa+T}\frac{2}{m_k(m_k+1)\eta_k}
\left( \left\|\m{x}-\breve{\m{x}}^{k}\right\|_{\C{H}}^2-\left\|\m{x}-
\breve{\m{x}}^{k+1}\right\|_{\C{H}}^2\right)
\leq \frac{2\left\|\m{x}- \breve{\m{x}}^\kappa\right\|_{\C{H}}^2}{m_\kappa(m_\kappa+1)\eta_\kappa}.
\end{equation}
Note that \[
\DELTA_t=\nabla f(\hat{\m{x}}_{t})-\m{d}_t=
\nabla f(\hat{\m{x}}_{t}) -\nabla f_{\xi_t}(\hat{\m{x}}_t)-\m{e}_t
\] only depends on the index $\xi_t$. So we have
  $\mathbb{E}[ \DELTA_t]=\m{0}$ since the random variable $\xi_t \in \{1,2,\ldots, N\}$
is chosen with uniform probability and $\mathbb{E}[\m{e}_t] = \m{0}$.
Also, since
$\breve{\m{x}}_t$ depends on $\xi_{t-1}$, $\xi_{t-2}$, $\ldots$, we have
$
\mathbb{E} \left[ \langle \DELTA_t, \breve{\m{x}}_{t}-\m{x}\rangle \right] =
\m{0}.
$
By the assumption  that $\mathbb{E}(\|\DELTA_t\|_{\C{H}^{-1}}^2) \le \sigma^2$, we have
\[
\mathbb{E} \left[ \sum_{t=1}^{m_k} t^2 \left\|\DELTA_t \right\|_{\C{H}^{-1}}^2 \right] \le
\frac{\sigma^2 m_k (m_k+1)(2m_k + 1)}{6} \le
 \frac{\sigma^2 }{2}  m_k^2 (m_k+1)
\]
since $m_k \ge 1$.
 Combining  these bounds for the terms in $\zeta^k$ with the condition   $\eta_k \le 1/(2\nu)$ is to get
\[
- \mathbb{E} \left[ \sum_{k=\kappa}^{\kappa+T} \zeta^k \right] \le
\frac{2}{m_\kappa(m_\kappa+1)\eta_\kappa} \left\|\m{x}- \breve{\m{x}}^\kappa\right\|_{\C{H}}^2
+ \frac{\sigma^2}{2} \sum_{k=\kappa}^{\kappa+T} \eta_k m_k .
\]
  Finally, applying  the expectation operator to
(\ref{Sec3-bjc-5}) and substituting this bound into the $\zeta^k$ term   complete the proof.
\end{proof}

 \smallskip
By properly setting the algorithm parameters, the following theorem shows the convergence
rate of  Algorithm~\ref{algo1} in the expectation of both the objective function value gap and the constraint violation.
\begin{theorem} \label{2-sec-bj1}
Suppose the conditions in  {Theorem}   \ref{Sec3-theore1}  hold. Let
\begin{equation}\label{Sec33-12ba2}
~~\eta_k= \min \left\{ \frac{c_1}{m_k(m_k+1)},  c_2  \right\} \quad \textrm{and} \quad m_k=\max\left\{\lceil c_3 k^{\varrho}\rceil, m\right\},
\end{equation}
where $c_1, c_2, c_3>0$, $\varrho \ge 1$ are constants and $m>0$ is a given integer.
Then, for every $\m{w}^*\in\Omega^*$, we have
\begin{equation}\label{Ex-Fbj}
\left| \mathbb{E} \big[ F(\m{w}_T)-F(\m{w}^*)\big]\right| = E_\varrho(T) =
\mathbb{E} \big[\left\|A\m{x}_{T}+B\m{y}_{T}-\m{b} \right\|\big],
\end{equation}
where $E_\varrho(T) = \C{O} (1/T)$ for $\varrho > 1$ and
$E_\varrho(T) = \C{O} (T^{-1}\log T)$ for $\varrho = 1$.
\end{theorem}
\begin{proof}
The proof is same as that of  \cite[Theorem 4.2]{BaiHz2019} and thus is omitted here.
\end{proof}

\begin{remark}\label{41re}
(I) In practice,  the matrix $\C{M}_k$ in Algorithm \ref{algo1}  could be adaptively adjusted as
$\C{M}_k=\rho_k\m{I}$, where $\rho_k = \max \{ \rho_{\min},   \beta \delta_2^k/\delta_1^k\}$
with $\rho_{\min}>0$,
\[
\delta_1^k = \left\| \m{x}^k - \m{x}^{k-1}  \right\|^2 \quad \mbox{and} \quad
\delta_2^k = \left\| A (\m{x}^k - \m{x}^{k-1})  \right\|^2.
\]
Since $\delta_2^k/\delta_1^k$ is an underestimate of
the largest eigenvalue of $ A \tr A$, to ensure convergence, the safeguard lower bound $\rho_{\min}$
should be increased during the optimization if necessary.
One may see \cite[Remark 4.2]{BaiHz2019} for more details.

{(II) When the set $\C{Y}$ is  {bounded}, we may even use a positive-{indefinite} proximal matrix
 \[
 L= \tau \chi \m{I}-  \beta B \tr B,\quad \textrm{where }~ \beta \|B \tr B \| \leq\chi<+\infty \mbox{ and } \tau \in [-1,1],
 \]
 in the update of $\m{y}$-subproblem. In this case, denoting
$
\C{N}_{\C{Y}}=\sup\{ \|\m{y}_1- \m{y}_2\|:   \m{y}_1, \m{y}_2\in \C{Y}\} < \infty,
$
analogous to   {Theorem} \ref{Sec3-theore1}  we can show
 \begin{eqnarray}\label{new-Ex-F}
&& \mathbb{E}\left[ F(\m{w}_T)
-F(\m{w})+(\m{w}_T-\m{w})\tr \mathcal{J}(\m{w})\right]\\
&&\leq
\frac{1}{2T}\bigg\{ \sigma^2\sum\limits_{k=\kappa}^{\kappa+T} \eta_k m_k
+ \frac{4}{m_\kappa(m_\kappa+1)\eta_\kappa} \left\|\m{x}- \breve{\m{x}}^\kappa\right\|_{\C{H}}^2 + \left\|\m{w}-\m{w}^\kappa\right\|^2_{\tilde{Q}_\kappa} \nonumber\\
&& ~   + \omega_1 \left\|A\m{x}^\kappa+B\m{y}^\kappa-\m{b}\right\|^2+ \omega_2\textbf{const}  \bigg\}, \nonumber
\end{eqnarray}
where
\[
\textbf{const}=\left \{\begin{array}{lll}
 \C{N}_{\C{Y}}^2(\beta\|B\|^2-\tau\chi), &\textrm{if }~ \tau\in[-1,0],\\
 \beta(\C{N}_{\C{Y}}\|B\|)^2+\tau\chi \left\|\m{y}^{\kappa-1}-\m{y}^\kappa\right\|^2,&\textrm{if }~ \tau\in(0,1].
\end{array}\right.
\]
This means that the results of Theorem \ref{2-sec-bj1} could still hold even when the proximal matrix $L$ is positive-indefinite.

{(III) Similar ideas of Algorithm~\ref{algo1} can be further generalized to solve separable convex optimization with one or multi-block structures.
For content focus of the paper, we leave these discussions in the Appendix.}
}
\end{remark}

\section{Numerical experiments}\label{NumExp}
In this section, we {apply the proposed algorithm} to solve the following graph-guided fused lasso problem
 in machine learning:
\[
\min\limits_{\mathbf{x}}\frac{1}{N}\sum\limits_{j=1}^{N}f_j(\mathbf{x})+\mu\|A\mathbf{x}\|_1,
\]
where  $f_j(\mathbf{x})=\log\left(1+\exp(-{b_ja_j\tr} \mathbf{x})\right)$ denotes the logistic loss function on the feature-label pair $(a_j,b_j)\in \mathbb{R}^{l}\times\{-1,1\}$,  $N{(> l)}$ is the data size , $\mu>0$ is a given  regularization parameter, and $A=[\mathbf{G};\mathbf{I}]$ is a matrix encoding the feature sparsity pattern. Here, $\mathbf{G}$ is the sparsity pattern of the graph that is obtained by sparse inverse covariance estimation \cite{FHTR08}.
Introducing an auxiliary variable $\m{y}$, the above problem is equivalent to the problem
\begin{equation} \label{Sec5-prob}
\begin{array}{lll}
\min\limits_{\mathbf{x,y}}&F(\mathbf{x},\mathbf{y}):=\frac{1}{N}\sum\limits_{j=1}^{N}f_j(\mathbf{x})+\mu\|\mathbf{y}\|_1\\
\textrm{s.t. } &   A\mathbf{x}-\mathbf{y}=\mathbf{0},\\
\end{array}
\end{equation}
which has the format of our {model   (\ref{P}).}
In addition, it can be easily verified that the Assumptions~\ref{asum-1}-\ref{asum-2} hold.
Since the coefficient matrix of the $\m{y}$ variable  in the constraints of (\ref{Sec5-prob})
is $- \m{I}$, the $\m{y}$-subproblem will have a closed-form solution by
simply setting $L= \m{0}$ in   Algorithm~\ref{algo1}.
Otherwise, the linearization techniques  discussed in (\ref{prox-y-problem}) on choosing
 $L$ can be applied to obtain a closed-form solution of the $\m{y}$-subproblem.
With $L= \m{0}$, the subproblems in Algorithm~\ref{algo1} would have the following
closed-form solution:
\begin{equation} \label{Sec5-subxy}
\left \{\begin{array}{lll}
\breve{\m{x}}_{t+1}&=&\left[\gamma_t \C{H}+ \C{M}_k\right]^{-1}\left[\gamma_t \C{H}\breve{\m{x}}_{t} + \C{M}_k \mathbf{x}^k - \m{d}_t- \m{h}^k\right],\\
\mathbf{y}^{k+1}&=&\textrm{Shrink}\left(\frac{\mu}{\beta},  A\mathbf{x}^{k+1}-\frac{\LAMBDA^{k+\frac{1}{2}}}{\beta}\right).
\end{array}\right.
\end{equation}
Here, $\textrm{Shrink}(\cdot,\cdot)$ denotes the soft shrinkage operator and can be
evaluated using the MATLAB built-in function ``\verb"wthresh"''.

In the numerical experiments, the penalty parameter in SAS-ADMM is taken as $\beta = 0.001$,
{the matrices $ \C{M}_k$ are updated adaptively by the strategy explained
in Remark 4.1 (I) with initial values $\rho_0 = 1$,
$\rho_{\min} = 10^{-5}$ and  $\C{H}=2\times 10^{-5}\m{I}$.
The other  parameters as well as the vector $\m{e}_t$  in SAS-ADMM (i.e. Algorithm~\ref{algo1})
 are chosen the same way as that used in \cite[Section 7.1]{BaiHz2019}, that is
\begin{equation}\label{et-set}
\m{e}_t =
\left \{\begin{array}{ll}
\nabla f(\m{x}_{k-1}) - \nabla f_{\xi_t} (\m{x}_{k-1})
& \mbox{if } ~ m_k >  l, \\
\m{0}, & \mbox{otherwise},
\end{array}\right.
\end{equation}
where $\m{x}_k$ is the ergodic mean of the $\m{x}$-iterates.
 Motivated from Theorem \ref{2-sec-bj1}}, we use
\[
\mbox{Obj\_err} = \frac{|F(\m{x}, \m{y})-F^*|}{\max\{F^*,1\}} \qquad \mbox{and} \qquad
\mbox{Equ\_err} = \left\| A \m{x} -\m{y} \right\|
\]
to denote the relative objective value error and the constraint violation error.
Here, $F^*$ is the approximate optimal objective function value
obtained by running Algorithm \ref{algo1} for more than $10$ minutes.
To measure the performance of a algorithm, we plot the maximum of
the relative objective error and the constraint error, that is
\[
\mbox{Opt\_err}=\max(\mbox{Obj\_err},\mbox{Equ\_err}),
\]
against the CPU time used.
 All experiments are implemented in MATLAB R2018a (64-bit)
with the same starting point
$(\mathbf{x}^0,\mathbf{y}^0,\LAMBDA^0)=(\mathbf{0},\mathbf{0},\mathbf{0})$
and   performed on a PC with Windows 10 operating system,
with an Intel i7-8700K CPU  and  16GB RAM.

\begin{figure}[htbp]
 \begin{minipage}{1\textwidth}
 \def\figurename{\footnotesize Fig.}
 \centering
 \resizebox{12cm}{5cm}{\includegraphics{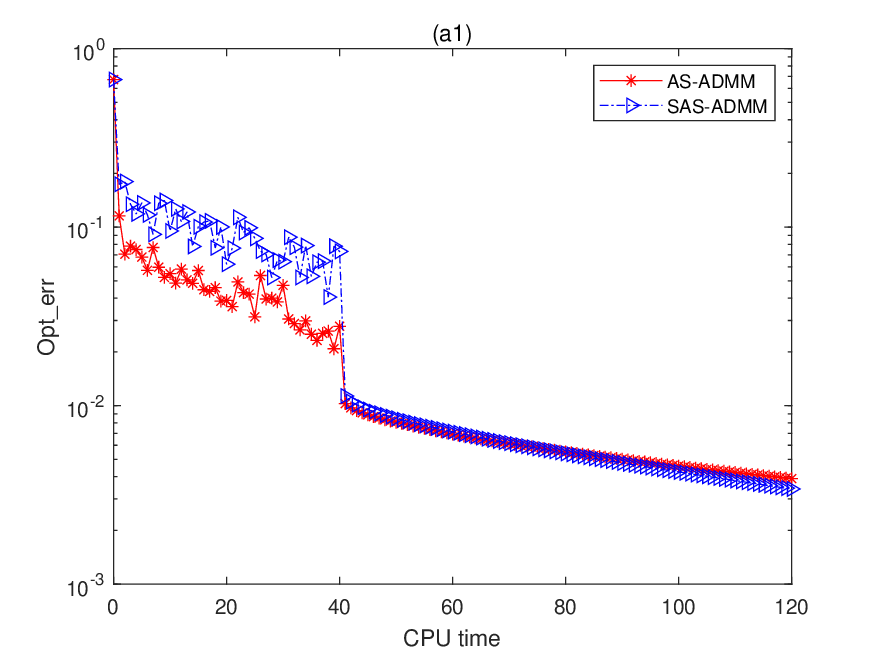}\includegraphics{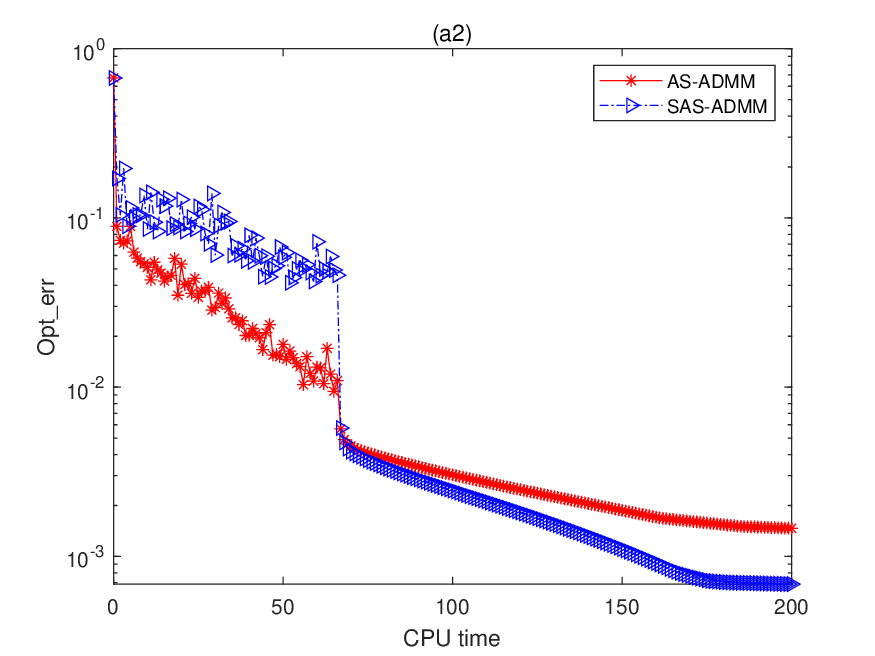}}
\resizebox{12cm}{5cm}{\includegraphics{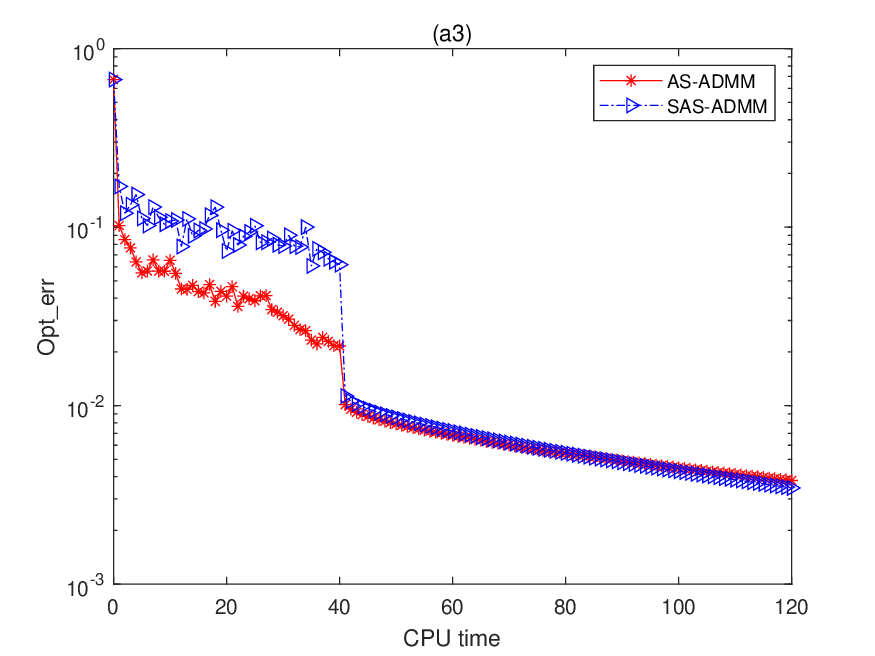}\includegraphics{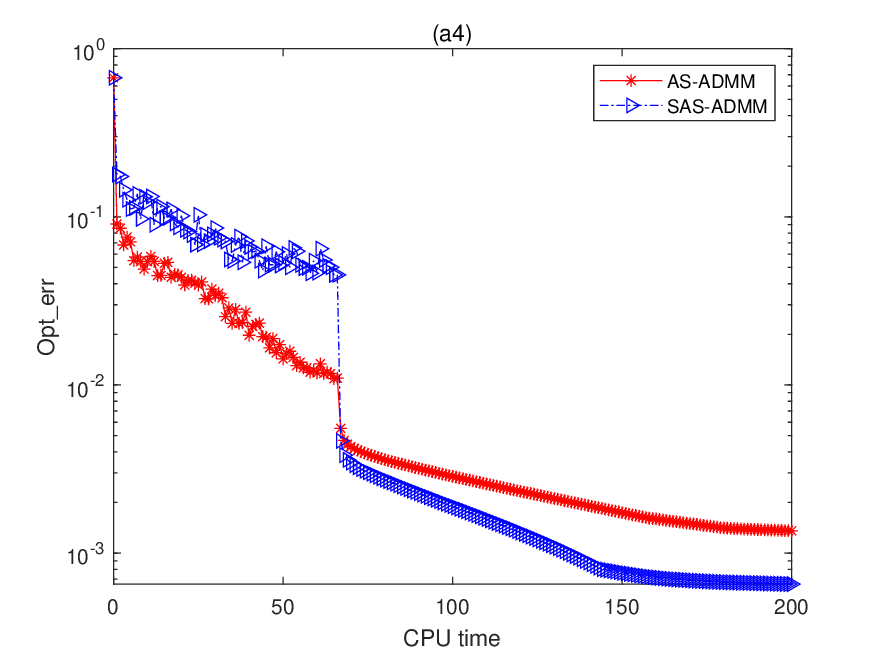}}
\caption{\footnotesize  Comparison of $\mbox{Opt\_err}$ vs CPU time
for  Problem (\ref{Sec5-prob}) on the  {\emph{mnist}} dataset: (a1)-(a2) are the results after $10$ successive runs; (a3)-(a4) are the results after $20$ successive runs} \label{FigProb1-1}
   \end{minipage}
\end{figure}

 \begin{figure}[htbp]
 \begin{minipage}{1\textwidth}
 \def\figurename{\footnotesize Fig.}
 \centering
 \resizebox{13.5cm}{4cm}{\includegraphics{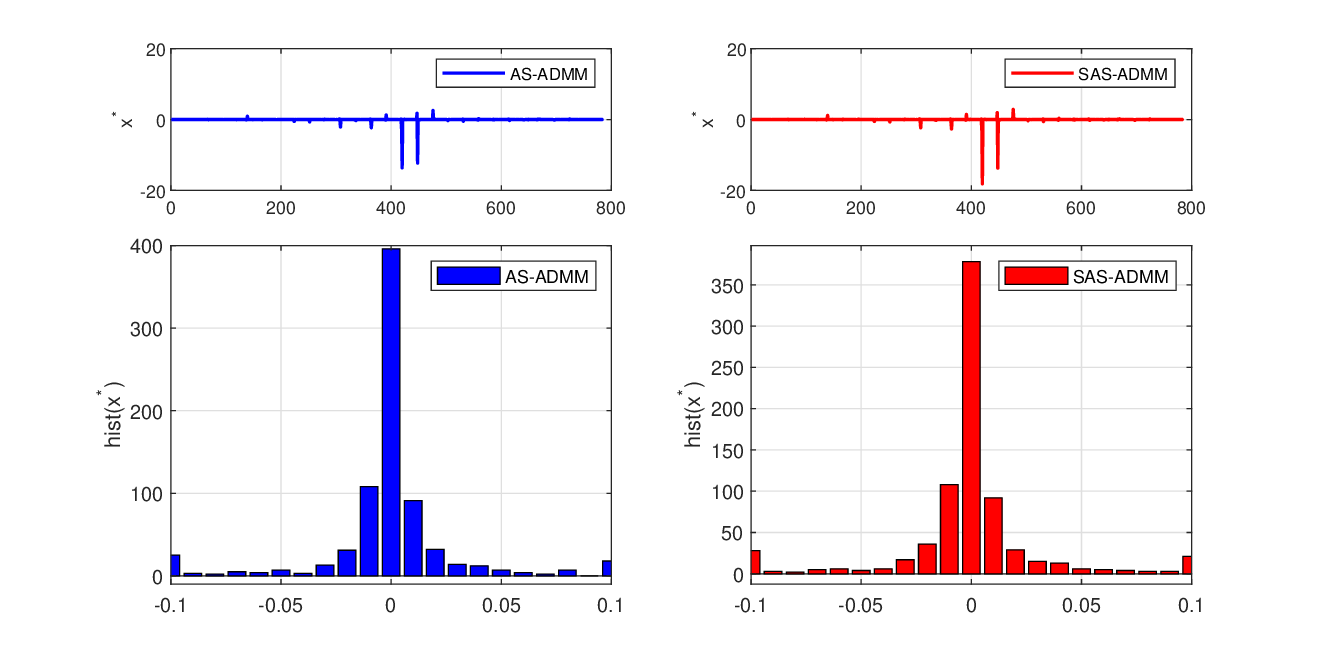}\includegraphics{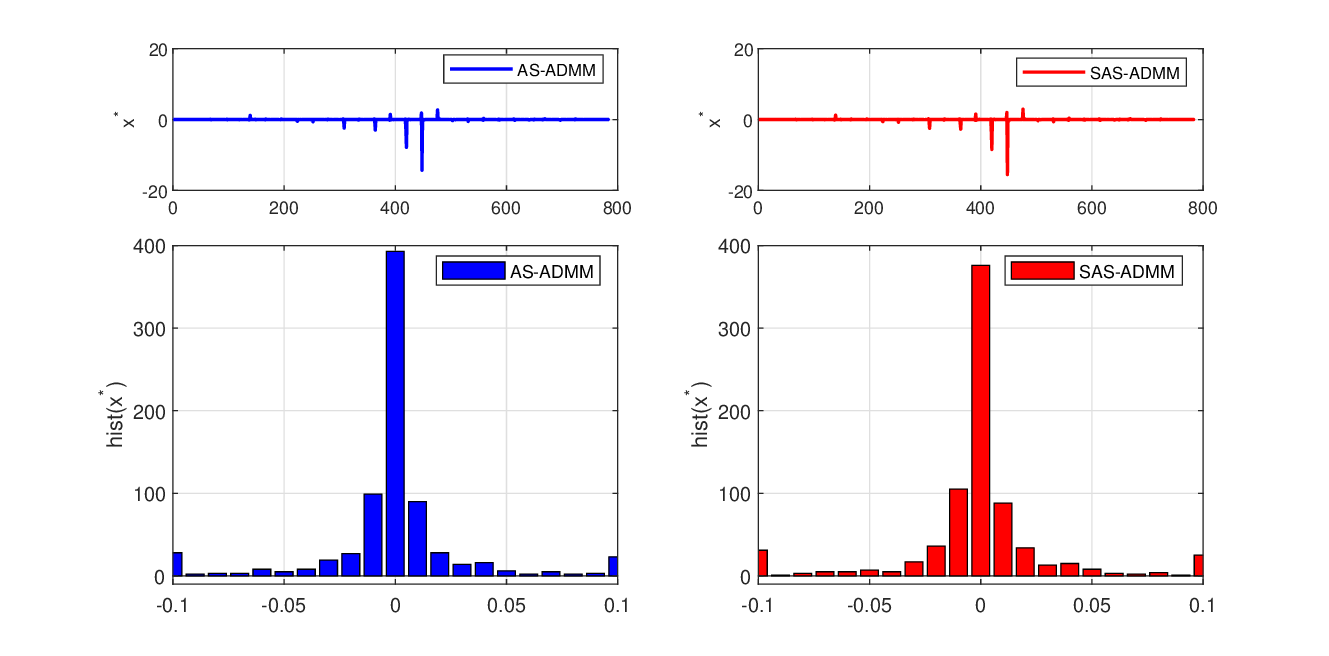}}
 \caption{\footnotesize  Comparison of the finally obtained iterate $\m{x}^{k+1}$ and $\textrm{hist}(\m{x}^{k+1})$ after $10$ successive runs: the left two  subfigures correspond  to (a1); the right two  subfigures  correspond  to (a2)} \label{FigProb1-b1}
   \end{minipage}
\end{figure}

 \begin{figure}[htbp]
 \begin{minipage}{1\textwidth}
 \def\figurename{\footnotesize Fig.}
 \centering
 \resizebox{13.5cm}{4cm}{\includegraphics{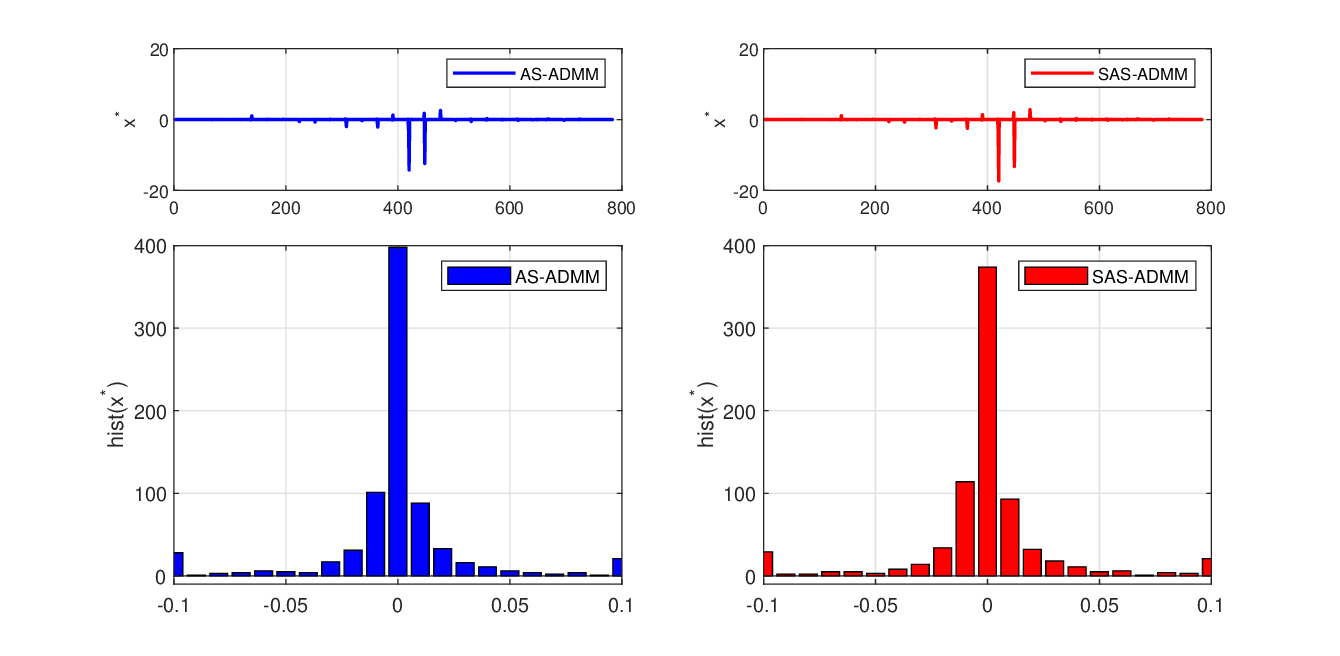}\includegraphics{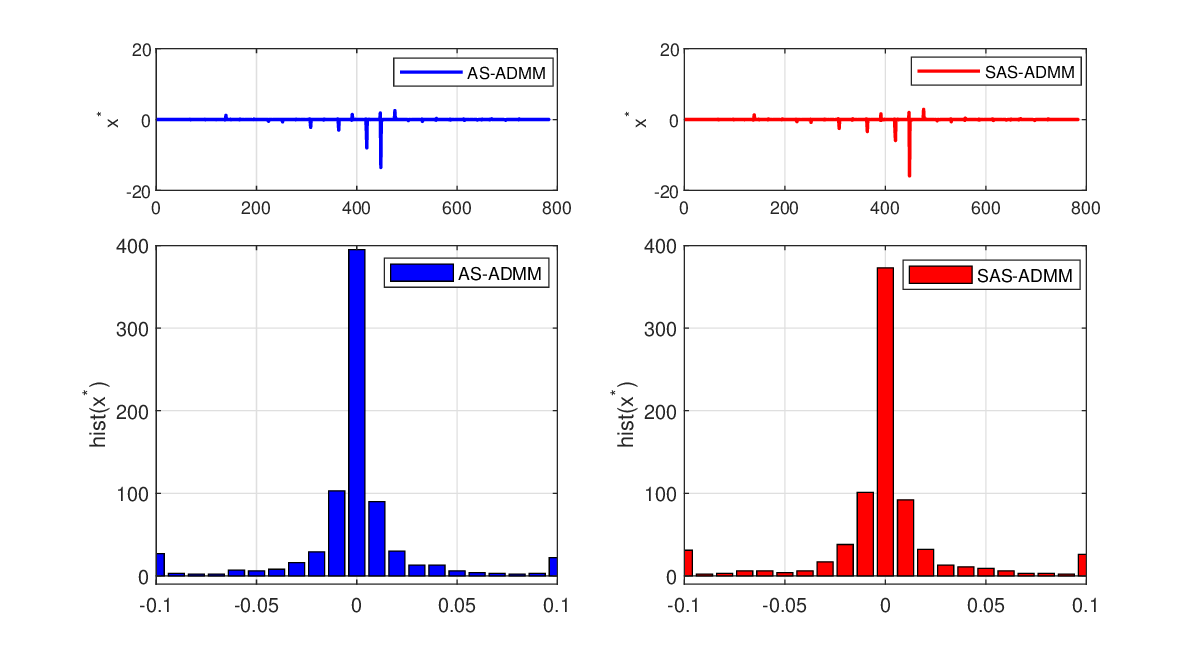}}
   \caption{\footnotesize   Comparison of the   finally obtained iterate $\m{x}^{k+1}$ and $\textrm{hist}(\m{x}^{k+1})$ after $20$ successive runs: the left two   subfigures  correspond  to (a3); the right two   subfigures  correspond  to (a4)} \label{FigProb1-b2}
    \end{minipage}
\end{figure}

We compare the numerical performance of the proposed algorithm SAS-ADMM\footnote{All codes are available at https://github.com/bjc1987/bjc1987.github.io
} using
stepsizes $(\tau,s)=(0.9,1.09)$, which is suggested in \cite{BLXZ2018} for GS-ADMM,
and AS-ADMM \cite{BaiHz2019}
for solving  problem (\ref{Sec5-prob}) on the
 dataset \emph{mnist} (including 11,791 samples and 784 features, {that is, $(N,l)=(11791,784)$}) downloaded from
 LIBSVM website.
 The regularization parameter $\mu$ in (\ref{Sec5-prob}) is set as $10^{-5}$.
 For both SAS-ADMM and AS-ADMM, we plot the error associated
with the iterates over the first 1/3 of the total CPU time budget, followed by the
error associated with the ergodic iterates over the last 2/3 of the budget.
 We { make $10$ and  $20$ successive runs of each algorithm under the CPU time budgets $120s$ and $200s$, respectively.
The average comparison results on $\mbox{Opt\_err}$ are shown in Figure  \ref{FigProb1-1}, and the comparison of the finally obtained iterative solution $\m{x}^{k+1}$ and $\textrm{hist}(\m{x}^{k+1})$  are shown in  Figures \ref{FigProb1-b1}-\ref{FigProb1-b2}.
 Here, we only compare SAS-ADMM with AS-ADMM since in \cite{BaiHz2019}  AS-ADMM was shown competitive or better than
other state-of-the-art deterministic and stochastic methods.}
Note that $\mbox{Opt\_err}$ has a big drop at around 1/3 of the CPU time budget, the
point where the ergodic iterates are started to use for reporting the objective value. From Figure \ref{FigProb1-1}, we can see that SAS-ADMM initially performs worse than AS-ADMM at the beginning iterations. But after the first $1/3$ of the total CPU time budget, the SAS-ADMM eventually seems to perform
better than AS-ADMM. {Finally, Figures \ref{FigProb1-b1}-\ref{FigProb1-b2} show that both
the comparison algorithms indeed get sparse solutions.}

\section{Conclusion}
We {proposed} a symmetric accelerated stochastic
alternating direction method of multipliers, called SAS-ADMM, whose dual variables are symmetrically
updated. We gave the specific dual stepsizes region
ensuring the global convergence, which is larger than those in the literature.
Under proper choice of the algorithm parameters, we proved the convergence of  SAS-ADMM in expectation with the worst-case $O(1/T)$ convergence rate, where $T$  represents  the number of iterations.
Our preliminary experiments showed that by symmetrically updating the dual variables
using a more flexible region, SAS-ADMM could outperform AS-ADMM, which only updates the
dual variable once, for solving some
structured optimization problems arising in machine learning.

\section{Appendix: further discussions}
In this section, we discuss   3-block extensions of Algorithm~\ref{algo1} and
its variance of a stochastic augmented Lagrangian method.

\subsection{A stochastic ALM}
We first consider a stochastic augmented Lagrangian method, a variant of SAS-ADMM, to solve
\begin{equation} \label{ALM-Prob}
\min \{ f(\m{x})  \mid
A\m{x} =\m{b},~ \m{x}\in \C{X} \},
\end{equation}
where  $\C{X} \subset \mathbb{R}^{n_1}$ is a closed  convex  subset,
and $f$ is an average of $N$ smooth convex functions as defined in (\ref{P}).
Now, the augmented Lagrangian of (\ref{ALM-Prob}) is
\[
\C{L}_\beta\left( \m{x}, \LAMBDA\right):  =  \C{L}\left( \m{x},\LAMBDA\right)
+\frac{\beta}{2} \left\| A\m{x}-\m{b}\right\|^2,
\]
where  $\C{L}\left( \m{x},\LAMBDA\right)= f(\m{x})- \LAMBDA\tr(A\m{x} -\m{b})$.
Then, based on Algorithm \ref{algo1}, we can propose the following
Accelerated Stochastic ALM (AS-ALM), Algorithm \ref{algo1-ALM}).
Similar to SAS-ADMM,  we can easily establish the following lemmas   on AS-ALM.
However, in this case, the convergence region for the dual stepsize
can be enlarged from $(0, (\sqrt{5}+1)/2]$ of AS-ADMM \cite{BaiHz2019} to $(0,2]$.

\renewcommand\figurename{Alg.}
\begin{figure}
\begin{center}
{ \tt
\begin{tabular}{l}
\hline
\\
{\bf Parameters:}
$\beta>0, s\in (0,2]$ and $  \C{H} \succ \m{0}$.\\
{\bf Initialization:}
$(\m{x}^0,\LAMBDA^0)$
$\in\C{X}\times \mathbb{R}^n:=\Omega$ and $ {\breve{\m{x}}^0=\m{x}^0}$.\\
For  $k=0,1,\ldots$ \\
\hspace*{.3in}Choose $m_k >0$, $\eta_k>0$ and
$ \C{M}_k $ such that $\C{M}_k - \beta A \tr A  \succeq \m{0}$.\\
\hspace*{.3in}$\m{h}^k :=$
$-A\tr \left[\LAMBDA^k-\beta(A\m{x}^k-\m{b})\right]$. \\
\hspace*{.3in}$(\m{x}^{k+1}, \breve{\m{x}}^{k+1}) =$
{\bf xsub} ($\m{x}^k, \breve{\m{x}}^k, \m{h}^k)$ with {\bf xsub} given  in ALG.\ref{algo1}. \\
\hspace*{.3in}$\LAMBDA^{k+1}=\LAMBDA^{k}-
s \beta\left(A\m{x}^{k+1}-\m{b}\right).$\\
end\\
\hline
\end{tabular}
}
\end{center}
\caption{ Accelerated  Stochastic ALM (AS-ALM)}\label{algo1-ALM}
\end{figure}
\renewcommand\figurename{Fig.}

\begin{lemma} \label{Sec31-3-ALM}
Let $\{ \m{x}^k \}$ be  generated by Algorithm~\ref{algo1-ALM} and $\eta_k \in (0, 1/\nu)$.
Then, the inequality (\ref{Sec3-5}) holds with
\begin{equation}\label{AS-ALM-tilde-lambda}
 \widetilde{\LAMBDA}^k=\LAMBDA^k-\beta \left( A \m{x}^{k+1}  - \m{b} \right).
\end{equation}
\end{lemma}

For the iterates generated by Algorithm~\ref{algo1-ALM}, in this subsection let
$\m{w}^k=(\m{x}^k; \LAMBDA^k)$ and $ \widetilde{\m{w}}^k=(\m{x}^{k+1}; \widetilde{\LAMBDA}^k)$,
where $ \widetilde{\LAMBDA}^k$ is defined in (\ref{AS-ALM-tilde-lambda}).
Then, we have the following lemma.
\begin{lemma} \label{Sec31-bz6-ALM}
Let   $\{ \m{w}^k \}$ be  generated by Algorithm~\ref{algo1-ALM} and $\eta_k \in (0, 1/\nu)$.
Then,  we have $ \widetilde{\m{w}}^k \in \Omega$ and
\[
f(\m{x})-f(\m{x}^{k+1}) +
\left\langle \m{w} - \widetilde{\m{w}}^k, \C{J}(\m{w})\right\rangle
\geq  ( \m{w}-\widetilde{\m{w}}^k)\tr Q_k (\m{w}^k-\widetilde{\m{w}}^k) + \zeta^k
 \]
for all $\m{w} \in \Omega$,  where $\zeta^k$ is given by (\ref{zeta_k}),
\[
\C{J}(\m{w})=\left(\begin{array}{c}
  -A\tr  \LAMBDA  \\   A\m{x} -\m{b}
\end{array}\right) \quad \mbox{ and } \quad Q_k=\left[\begin{array}{ccc}
           \C{D}_k  &&  \\
            & &    \frac{1}{\beta} \m{I}
\end{array}\right].
\]
\end{lemma}
\begin{proof}
Combining the inequality (\ref{Sec3-5}) and the   relation
$
 A \m{x}^{k+1}  - \m{b} =\frac{1}{\beta}(\LAMBDA^k-\widetilde{\LAMBDA}^k)
$
gives the results.
\end{proof}
\begin{lemma} \label{key-lemma-ALM}
Let $\{ \m{w}^k \}$ be  generated by Algorithm~\ref{algo1-ALM} and $\eta_k\in(0,  {1}/{\nu})$.
Then, for any $\m{w} \in \Omega$,  we have
\begin{eqnarray*}
&&
f(\m{x})-f(\m{x}^{k+1}) + ( \m{w} - \widetilde{\m{w}}^k)\tr \C{J}(\m{w})\\
& &\ge
\frac{1}{2}\left\{\left\|\m{w}-\m{w}^{k+1}\right\|^2_{\widetilde{Q}_k}-\left\|\m{w}-\m{w}^{k}\right\|^2_{\widetilde{Q}_k}+  \left\|\m{w}^k-\widetilde{\m{w}}^k \right\|_{\widetilde{G}_k}^2 \right\} + \zeta^k,
\end{eqnarray*}
where $\zeta^k$ is given by (\ref{zeta_k}),
\[
\widetilde{Q}_k=\left[\begin{array}{ccc}
           \C{D}_k  &&  \\
            & &    \frac{1}{s \beta} \m{I}
\end{array}\right] \quad \mbox{ and } \quad
\widetilde{G}_k=\left[\begin{array}{ccc}
           \C{D}_k  &&  \\
            & &    \frac{2-s}{ \beta} \m{I}
\end{array}\right].
\]
\end{lemma}
\begin{proof}
The proof is similar to that of Corollary \ref{coll37} and is omitted.
\end{proof}

Finally, under the   conditions of Theorem~\ref{2-sec-bj1},   by {Lemma}~\ref{key-lemma-ALM} and
a similar proof of Theorem~\ref{2-sec-bj1}, we can deduce that for any
$\m{w}_T :=\frac{1}{1+T}\sum_{k=\kappa}^{\kappa+T}\tilde{\m{w}}^{k}$ and $\kappa\geq 0$, it has
   \[
 \left|\mathbb{E}\left[ f(\m{x}_T)-f(\m{x}^*)\right]\right| = E_\varrho(T) =
\mathbb{E} \big[\left\|A\m{x}_{T} -\m{b} \right\|\big].
  \]
where $E_\varrho(T) = \C{O} (1/T)$ for $\varrho > 1$ and
$E_\varrho(T) = \C{O} (T^{-1}\log T)$ for $\varrho = 1$.

\subsection{Three-block extensions} \label{diss-sec}

Consider {a} 3-block extension of problem (\ref{P})
\begin{equation} \label{Sec33-Prob}
\begin{array}{lll}
\min   & F(\m{w}):=f(\m{x})+g(\m{y})+l(\m{z})\\
\textrm{s.t. } &\C{K}\m{w}:=A\m{x}+B\m{y}+ C\m{z}=\m{b},\\
& \m{x}\in \C{X},\ \m{y}\in \C{Y},\ \m{z}\in \C{Z},
\end{array}
\end{equation}
where  $ l$ is  a closed convex function,  $C\in \mathbb{R}^{n \times n_3}$ is a given   matrix,
$\C{Z} \subset \mathbb{R}^{n_3}$ is a  simple  closed  convex  subset,  and the other functions and variables
remain the same definitions as those in problem (\ref{P}).
{ Here, the additional function $l$ can be possibly} used to
promote some data structure different from the structure promoted by $g$.
For convenience, in this subsection, let us define $\C{K}\m{w}:= A\m{x}+B\m{y}+ C\m{z}$,
denote $\m{w}=(\m{z};\m{x};\m{y}; \LAMBDA)$,
\begin{equation}\label{Sec4-vector-w-v1211}
\quad~~\m{w}^k=\left(\begin{array}{c}
 \m{x}^k\\  \m{y}^k \\ \m{z}^k\\    \LAMBDA^k
\end{array}\right),
  \widetilde{\m{w}}^k:=\left(\begin{array}{c}
   \widetilde{\m{x}}^k\\ \widetilde{\m{y}}^k\\\widetilde{\m{z}}^k\\
   \widetilde{\LAMBDA}^{k}
\end{array}\right)=
\left(\begin{array}{c}
   \m{x}^{k+1}\\ \m{y}^{k+1}\\ \m{z}^{k+1}\\
  \widetilde{\LAMBDA}^{k}
\end{array}\right) \mbox{ and }
\C{J}(\m{w})=\left(\begin{array}{c}
-A\tr \LAMBDA\\ -B\tr\LAMBDA\\ -C\tr \LAMBDA\\ \C{K}\m{w}-\m{b}
\end{array}\right),
\end{equation}
where $\widetilde{\LAMBDA}^k$ will be specified differently in the following two discussion cases.

\subsubsection{Extension in  Gauss-Seidel update}
For this case, we need an assumption that $C \tr A = \m{0}$.
Then SAS-ADMM can be directly extended to Algorithm \ref{algo51} for solving the 3-block problem
 (\ref{Sec33-Prob}), where the variable updating order is $\m{z}^{k+1}\rightarrow \m{x}^{k+1}\rightarrow\m{y}^{k+1}\rightarrow\LAMBDA^{k+1}$ in a Gauss-Seidel scheme. Now, let{
\begin{eqnarray}
&&  \widetilde{\LAMBDA}^{k}= \LAMBDA^k- \beta \left(C\m{z}^{k+1}+A\m{x}^{k+1}+B\m{y}^{k}-\m{b}\right), \label{Sec51-0bz1}\\
&& \m{v}^{k+1}=(\m{x}^{k+1};\m{y}^{k+1};  \LAMBDA^{k+1})~~
\mbox{ and }~~ \widetilde{\m{v}}^k=(\widetilde{\m{x}}^k; \widetilde{\m{y}}^k; \widetilde{\LAMBDA}^{k}). \label{Sec51-vk}
\end{eqnarray}
Then,} we have the following main {lemma} for the convergence of Algorithm~\ref{algo51}.
\renewcommand\figurename{Alg.}
\begin{figure}
\begin{center}
{ \tt
\begin{tabular}{l}
\hline
\\
{\bf Parameters:}
$\beta>0, ~ \C{H} \succ \m{0},~ L\succeq \m{0} $ and $(\tau,s)\in \Delta$.\\
{\bf Initialization:}
${(\m{x}^0,\m{y}^0,\m{z}^0,\LAMBDA^0) \in\C{X}\times\C{Y}\times \C{Z}\times}\mathbb{R}^n:=\Omega,~ {\breve{\m{x}}^0=\m{x}^0}$.\\
For  $k=0,1,\ldots$ \\
\hspace*{.3in}Choose $m_k>0$, $\eta_k > 0$ and
$ \C{M}_k $ such that $\C{M}_k - \beta A \tr A  \succeq \m{0}$.\\
\hspace*{.3in}$\m{z}^{k+1}\in\arg\min\limits_{\m{z} \in \C{Z}} l(\m{z})+ \frac{\beta}{2}
\left\| C\m{z}+ A\m{x}^{k}+B\m{y}^k-\m{b}-\frac{\LAMBDA^{k}}{\beta}
\right\|^2 $.\\
\hspace*{.3in}$\m{h}^k :=$
$-A\tr \left[\LAMBDA^k-\beta(C\m{z}^{k+1}+ A\m{x}^k+B\m{y}^k-\m{b})\right]$. \\
\hspace*{.3in}$(\m{x}^{k+1}, \breve{\m{x}}^{k+1}) =$
{\bf xsub} ($\m{x}^k, \breve{\m{x}}^k, \m{h}^k)$ with {\bf xsub} given in ALG.\ref{algo1}. \\
\hspace*{.3in}$\LAMBDA^{k+\frac{1}{2}}=\LAMBDA^k-
\tau \beta\left(C\m{z}^{k+1}+A\m{x}^{k+1}+B\m{y}^{k}-\m{b}\right)$.\\
\hspace*{.3in}{\small$\m{y}^{k+1}\in\arg\min\limits_{\m{y} \in \C{Y}}  g(\m{y})+ \frac{\beta}{2}
\left\| C\m{z}^{k+1}+ A\m{x}^{k+1}+B\m{y}-\m{b}-\frac{\LAMBDA^{k+\frac{1}{2}}}{\beta}
\right\|^2 +\frac{1}{2}\left\|\m{y}-\m{y}^k\right\|_{L}^2   $}.\\
\hspace*{.3in}$\LAMBDA^{k+1}=\LAMBDA^{k+\frac{1}{2}}-
s \beta\left(A\m{x}^{k+1}+B\m{y}^{k+1}+C\m{z}^{k+1}-\m{b}\right).$\\
end\\
\hline
\end{tabular}
}
\end{center}
\caption{ Extension of SAS-ADMM in Gauss-Seidel update }\label{algo51}
\end{figure}
\renewcommand\figurename{Fig.}

\begin{lemma} \label{Sec51-bz2}
Assume $C\tr A=\m{0}$ and $\eta_k \in (0, 1/\nu)$. Then, the iterates
generated by Algorithm~\ref{algo51} satisfy
  $\widetilde{\m{w}}^k\in \Omega$ and
\begin{eqnarray*}
&&F(\m{w})-F(\widetilde{\m{w}}^k) +
\left( \m{w} - \widetilde{\m{w}}^k\right)\tr \C{J}(\m{w}) \\
&&\ge
\frac{1}{2}\left\{\left\|\m{v}-\m{v}^{k+1}\right\|^2_{\widetilde{Q}_k}-\left\|\m{v}-\m{v}^{k}\right\|^2_{\widetilde{Q}_k}+  \left\|\m{v}^k-\widetilde{\m{v}}^k \right\|_{\widetilde{G}_k}^2 \right\} + \zeta^k
\end{eqnarray*}
  for any $\m{w} \in \Omega$,  where   $\widetilde{Q}_k,\widetilde{G}_k$ and $\zeta^k$ are
   given in Corollary \ref{coll37} and (\ref{zeta_k}), respectively.
Moreover, we have
 \begin{eqnarray*}
&& \left\|\m{v}^k-\tilde{\m{v}}^k\right\|_{\tilde{G}_k}^2 \geq
\left\| \m{x}^k - \m{x}^{k+1}\right\|_{\C{D}_k}^2
 +\omega_0 \left\| A\m{x}^{k+1}+B\m{y}^{k+1}-\m{b} \right\|^2  \\
& &
+ \omega_1 \left(\left\| A\m{x}^{k+1}+B\m{y}^{k+1}-\m{b} \right\|^2 -\left\| A\m{x}^k+B\m{y}^k-\m{b} \right\|^2 \right) \\
&&+ \omega_2\left( \left\|\m{y}^k-\m{y}^{k+1}\right\|_L^2 -
 \left\|\m{y}^{k-1}-\m{y}^k\right\|_L^2 \right),
\end{eqnarray*}
where $\omega_0, \omega_1, \omega_2 \ge 0$ are given in (\ref{def_omega}).
\end{lemma}
\begin{proof}
By the updates of $\m{h}^k$ and $\widetilde{\LAMBDA}^k$ in Algorithm~\ref{algo51},
it is easy to derive (\ref{Sec3-5}) as before.
Then, according to the first-order optimality condition of $\m{z}$-subproblem and
the assumption that $C\tr A=\m{0}$, we have
\begin{equation}\label{Sec51-bz6}
 \m{z}^{k+1}\in \C{Z},\quad l(\m{z})-l(\m{z}^{k+1})+ \left\langle\m{z}-\m{z}^{k+1}, \m{p}^k_{\m{z}}\right\rangle\geq  0,\quad \forall \m{z}\in \C{Z},
\end{equation}
where
\begin{eqnarray}
\m{p}^k_{\m{z}}&=& -C\tr  \LAMBDA^{k} +\beta C\tr\left(C\m{z}^{k+1}+A\m{x}^{k}+B\m{y}^{k}-\m{b}\right)    \nonumber\\
  &=&   -C\tr \widetilde{\LAMBDA}^{k} -\beta C\tr A(\m{x}^{k+1}-\m{x}^{k})= -C\tr \widetilde{\LAMBDA}^{k}. \nonumber
\end{eqnarray}
Similarly, we have by the $\m{y}$-update that
\begin{equation}\label{Sec51-bz7}
 \m{y}^{k+1}\in \C{Y},\quad g(\m{y})-g(\m{y}^{k+1})+ \left\langle\m{y}-\m{y}^{k+1}, \m{p}^k_{\m{y}}\right\rangle\geq  0,\quad \forall \m{y}\in \C{Y},
\end{equation}
where
\begin{eqnarray*}
\m{p}^k_y&=&
-B\tr  \LAMBDA^{k+\frac{1}{2}} +\beta B\tr\left(\C{K} \m{w}^{k+1}-\m{b}\right)  +L(\m{y}^{k+1}-\m{y}^k)
\nonumber\\
&=&  -B\tr  \LAMBDA^{k+\frac{1}{2}} +  B\tr \left( \LAMBDA^k- \tilde{\LAMBDA}^k\right)+  \left[
L
+ \beta B\tr B \right] \left(\m{y}^{k+1} - \m{y}^k  \right)\\
&=& -B\tr \tilde{\LAMBDA}^k + \tau B\tr( \LAMBDA^{k}-\tilde{\LAMBDA}^{k}) +
 \left[ L+ \beta B\tr B \right] \left(\m{y}^{k+1} - \m{y}^k  \right).
\end{eqnarray*}
Besides, it follows from the updates of $\widetilde{\LAMBDA}^k$ that
\begin{eqnarray}\label{Sec51-bz9}
\left\langle\LAMBDA- \widetilde{\LAMBDA}^k, \C{K}\widetilde{\m{w}}^k-\m{b} +\frac{1}{\beta}\left(\widetilde{\LAMBDA}^k-\LAMBDA^k\right) -B\left(\widetilde{\m{y}}^{k}-\m{y}^k\right)\right\rangle =0,\quad \forall \LAMBDA\in \mathbb{R}^n.
\end{eqnarray}
Combining the above inequalities (\ref{Sec51-bz6}), (\ref{Sec51-bz7}), (\ref{Sec51-bz9})
with (\ref{Sec3-5}), we can get
\begin{equation} \label{Sec51-bz10}
F(\m{w})-F(\widetilde{\m{w}}^k) +
\left( \m{w} - \widetilde{\m{w}}^k\right)\tr \C{J}(\widetilde{\m{w}}^k)
\geq  (\m{v}  -\widetilde{\m{v}}^k)\tr Q_k (\m{v}^k-\widetilde{\m{v}}^k) + \zeta^k,
 \end{equation}
where $\zeta^k, Q_k$ are given by (\ref{zeta_k}) and (\ref{Sec31-vwuJ}), respectively.
Then, the rest proof will be similar to that of Corollary~\ref{coll37} and {{Lemma}}~\ref{key-lemma}.
\end{proof}

Based on the above {Lemma}~\ref{Sec51-bz2}, the ergodic convergence of Algorithm~\ref{algo51}
with a sublinear convergence rate can be similarly established under the conditions
of Theorem~\ref{2-sec-bj1}.
Here, we omit the detailed proof. {Note that, if the $\m{z}$-subproblem is not easily solvable,
one could also add a positive semidefinite proximal term to linearize it. However, the requirement  $C\tr A=\m{0}$
is quite strict in applications.
In the next subsection we will propose a partially Jacobi update for the primal variables, for which $C\tr A=\m{0}$ is not required.}

\subsubsection{Extension in partially  Jacobi update}
Now, let us consider Algorithm \ref{algo3}, where
the block variables $\m{y}$ and $\m{z}$ are updated in a Jacobi fashion.

\renewcommand\figurename{Alg.}
\begin{figure}
\begin{center}
{ \tt
\begin{tabular}{l}
\hline
\\
{\bf Parameters:}
  $\beta>0,  (\tau,s)\in \Delta$,
 $L_1 $ and  $L_2$ satisfy (\ref{L1-L2}).\\
{\bf Initialization:}
$(\m{x}^0,\m{y}^0,\m{z}^0,\LAMBDA^0)$
$\in\C{X}\times\C{Y}\times\C{Z}\times \mathbb{R}^n:=\Omega,~ {\breve{\m{x}}^0=\m{x}^0}$.\\
For  $k=0,1,\ldots$ \\
\hspace*{.3in}Choose $m_k>0$, $\eta_k >0$ and
$ \C{M}_k $ such that $\C{M}_k - \beta A \tr A  \succeq \m{0}$.\\
\hspace*{.3in}$\m{h}^k :=$
$-A\tr \left[\LAMBDA^k-\beta( A\m{x}^k+B\m{y}^k+C\m{z}^k-\m{b})\right]$. \\
\hspace*{.3in}$(\m{x}^{k+1}, \breve{\m{x}}^{k+1}) =$
{\bf xsub} ($\m{x}^k, \breve{\m{x}}^k)$ with {\bf xsub} given  in ALG.\ref{algo1}. \\
\hspace*{.3in}$\LAMBDA^{k+\frac{1}{2}}=\LAMBDA^k-
\tau \beta\left(A\m{x}^{k+1}+B\m{y}^{k} +C\m{z}^k-\m{b}\right)$.\\
\hspace*{.3in}{\small$\m{y}^{k+1}\in\arg\min\limits_{\m{y} \in \C{Y}}  g(\m{y})+ \frac{\beta}{2}
\left\|   A\m{x}^{k+1}+B\m{y}+C\m{z}^k -\m{b}-\frac{\LAMBDA^{k+\frac{1}{2}}}{\beta}
\right\|^2 +\frac{1}{2}\left\|\m{y}-\m{y}^k\right\|_{L_1}^2   $}.\\
\hspace*{.3in}{\small$\m{z}^{k+1}\in\arg\min\limits_{\m{z} \in \C{Z}} l(\m{z})+ \frac{\beta}{2}
\left\| A\m{x}^{k+1}+B\m{y}^k +C\m{z}-\m{b}-\frac{\LAMBDA^{k+\frac{1}{2}}}{\beta}
\right\|^2+\frac{1}{2}\left\|\m{z}-\m{z}^k\right\|_{L_2}^2 $.}\\
\hspace*{.3in}$\LAMBDA^{k+1}=\LAMBDA^{k+\frac{1}{2}}-
s \beta\left(A\m{x}^{k+1}+B\m{y}^{k+1}+C\m{z}^{k+1}-\m{b}\right).$\\
end\\
\hline
\end{tabular}
}
\end{center}
\caption{ Extension of SAS-ADMM in partially Jacobi update }\label{algo3}
\end{figure}
\renewcommand\figurename{Fig.}

To establish the global convergence of Algorithm~\ref{algo3}, we first have the following
observations. {Denoting
 \begin{equation}\label{Sec4-Sec31-vwuJ}
 \widetilde{\LAMBDA}^k=
  \LAMBDA^{k}-\beta \left(A
\m{x}^{k+1}+B\m{y}^{k}+C\m{z}^{k}-\m{b}\right)
\end{equation}
and using} the first-order optimality condition of the $\m{y}$-subproblem, we have
\begin{equation}\label{Sec4-Chapt5-a-Sec31-xN}
 \m{y}^{k+1}\in \C{Y},\quad g(\m{y})- g(\m{y}^{k+1})+ \left\langle\m{y}-\m{y}^{k+1}, \m{p}^k_{\m{y}}\right\rangle\geq  0,\quad \forall \m{y}\in \C{Y},
\end{equation}
where
\begin{eqnarray}\label{Sec4-Chapt5-S-Sec32-RxN}
\m{p}^k_{\m{y}}&=& -B\tr  \LAMBDA^{k+\frac{1}{2}} +\beta B\tr\left(A\m{x}^{k+1}+B\m{y}^{k+1}+C\m{z}^{k}-\m{b}\right) +L_1\left(\m{y}^{k+1}-\m{y}^{k}\right)   \nonumber\\
&=& -B\tr  \LAMBDA^{k+\frac{1}{2}} +\beta B\tr\left(A\m{x}^{k+1}+B\m{y}^k+C\m{z}^{k}-\m{b}\right) +(L_1+\beta B\tr B)\left(\m{y}^{k+1}-\m{y}^{k}\right)   \nonumber\\
  &=&   -B\tr \widetilde{\LAMBDA}^{k}+\tau B\tr (\LAMBDA^k - \widetilde{\LAMBDA}^k)
  +(L_1 +\beta B\tr B)\left(\m{y}^{k+1}-\m{y}^{k}\right),\nonumber
\end{eqnarray}
and we use the relationship
\begin{equation}\label{sec5-LA12}
\LAMBDA^{k+\frac{1}{2}}= \LAMBDA^{k}-\tau( \LAMBDA^{k}-\tilde{\LAMBDA}^{k}).
\end{equation}
{combining (\ref{Sec4-Chapt5-a-Sec31-xN}) and the definition of $\m{p}^k_{\m{y}}$,} we have
\begin{eqnarray}\label{3block-bz1}
&&~~~~~g(\m{y})- g(\m{y}^{k+1})+ \left\langle \m{y}-\m{y}^{k+1},- B\tr \LAMBDA^{k+\frac{1}{2}} +\beta B\tr (\C{K}\m{w}^{k+1}-\m{b}) \right.\\
&&\qquad\qquad\qquad\qquad\qquad\left.  -\beta B\tr C( \m{z}^{k+1}-\m{z}^{k} ) +L_1(\m{y}^{k+1}-\m{y}^{k}) \right\rangle\geq  0.  \nonumber
\end{eqnarray}
Similarly, by the first-order optimality condition of the $\m{z}$-subproblem, we have
\begin{eqnarray}\label{3block-bz11}
&&~~~~~l(\m{z})- l(\m{z}^{k+1})+ \left\langle \m{z}-\m{z}^{k+1},- B\tr \LAMBDA^{k+\frac{1}{2}} +\beta C\tr (\C{K}\m{w}^{k+1}-\m{b})\right.\\
&&\qquad\qquad\qquad\qquad\qquad\left.  -\beta C\tr B( \m{y}^{k+1}-\m{y}^{k} ) +L_2(\m{z}^{k+1}-\m{z}^{k}) \right\rangle\geq  0.  \nonumber
\end{eqnarray}
Adding the above two inequalities (\ref{3block-bz1}) and (\ref{3block-bz11}), we can see
$(\m{y}^{k+1}, \m{z}^{k+1})$ satisfies the first-order {optimality} condition, hence
is a solution, of the following problem
\begin{eqnarray}\label{yz-subproblem}
(\m{y}^{k+1}, \m{z}^{k+1}) &\in\arg\min\limits_{\m{y} \in \C{Y}, \m{z} \in \C{Z}} &
g(\m{y})+ l(\m{z}) + \frac{\beta}{2} \left\| A\m{x}^{k+1}+B\m{y}+C\m{z} -\m{b}
-\frac{\LAMBDA^{k+\frac{1}{2}}}{\beta} \right\|^2  \nonumber \\
 && +\frac{1}{2}\left\|(\m{y}-\m{y}^k, \m{z}-{\m{z}^k})\right\|_{\bar{L}}^2,
\end{eqnarray}
where
\begin{equation}\label{bar-L}
\bar{L}=\left[\begin{array}{ccc}
           L_1 && - \beta B \tr C \\
           - \beta C\tr B&&  L_2
\end{array}\right].
\end{equation}
Hence, by considering $(\m{y}, \m{z})$ as one block variable,
Algorithm~\ref{algo3} is essentially a particular version of Algorithm~\ref{algo1}
for solving a 2-block problem with $L$ and $B$ being replaced by $\bar{L}$ and
$(B, C)$, respectively.

From the above observations, we can directly establish the following properties of
Algorithm~\ref{algo3}.
\begin{lemma} \label{Sec4-Sec31-bz6}
The iterates  generated by Algorithm~\ref{algo3} satisfy
\[
F(\m{w})-F(\tilde{\m{w}}^k) +
\left\langle \m{w} - \tilde{\m{w}}^k,
\C{J}(\m{w})\right\rangle
\geq  ( \m{w}-\tilde{\m{w}}^k)\tr
Q_k (\m{w}^k-\tilde{\m{w}}^k) + \zeta^k
\]
for any $\m{w} \in \Omega$, where $\zeta^k$ is given by (\ref{zeta_k}),
 \begin{equation}\label{3block-Qk}
Q_k=\left[\begin{array}{ccccccc}
           \C{D}_k && &&  &&\\
          &&L_1+\beta B\tr B&&   && -\tau B\tr\\
         && && L_2+\beta C\tr C  && -\tau C\tr\\
      && -B & & -C &&  \frac{1}{\beta} \m{I}
\end{array}\right].
\end{equation}
\end{lemma}
\begin{proof}
Notice that
\[
\bar{L} + \beta (B, C )\tr (B, C) =
\left[\begin{array}{ccc}
           L_1 + \beta B \tr B &&  \\
           &&  L_2 + \beta C \tr C
\end{array}\right].
\]
So, replacing $L$ and $B$ in {Lemma}~\ref{Sec31-bz6} by
$\bar{L}$ and $(B, C)$, respectively, this lemma
directly follows from {Lemma}~\ref{Sec31-bz6}.
\end{proof}

Similarly, identifying $L$ and $B$ in (\ref{tilde-Q})
with $\bar{L}$ and $(B, C)$, respectively, it follows from  Corollary \ref{coll37}
that
\begin{eqnarray} \label{bjc-39-141}
&&F(\m{w})-F(\widetilde{\m{w}}^{k}) + {\left\langle \m{w} - \widetilde{\m{w}}^k,  \C{J}(\m{w})\right\rangle} \\
&& \ge \frac{1}{2}\left\{\left\|\m{w}-\m{w}^{k+1}\right\|^2_{\widetilde{Q}_k}-\left\|\m{w}-\m{w}^{k}\right\|^2_{\widetilde{Q}_k}+  \left\|\m{w}^k-\widetilde{\m{w}}^k\right\|_{\widetilde{G}_k}^2 \right\} + \zeta^k,\nonumber
\end{eqnarray}
 where $\zeta^k$ is given by (\ref{zeta_k}) and
\begin{eqnarray}\label{3block-tidleQk1}
\qquad~ \widetilde{Q}_k=\left[\begin{array}{cccc}
           \C{D}_k & &  &\\
          &L_1+(1-\frac{\tau s}{\tau +s})\beta B\tr B& -\frac{\tau s}{\tau +s} \beta B\tr C  & -\frac{\tau  }{\tau +s} B\tr\\
         & -\frac{\tau s}{\tau +s}\beta C\tr B& L_2+(1-\frac{\tau s}{\tau +s})\beta C\tr C  & -\frac{\tau  }{\tau +s}C\tr\\
      &  -\frac{\tau  }{\tau +s} B   & -\frac{\tau  }{\tau +s} C  &  \frac{1}{(\tau +s)\beta} \m{I}
\end{array}\right],
\end{eqnarray}
\begin{eqnarray}\label{3block-tidleGk11}
~~\widetilde{G}_k=\left[\begin{array}{cccc}
           \C{D}_k & &  &\\
          &L_1+(1-s)\beta B\tr B& -s\beta B\tr C  & (s-1) B\tr\\
         & -s\beta C\tr B& L_2+(1-s )\beta C\tr C  &(s-1) C\tr \\
      &  (s-1) B  &  (s-1) C &  \frac{2-\tau-s}{\beta} \m{I}
\end{array}\right].
\end{eqnarray}
Then, we have the following {estimate} on a lower bound of $\left\|\m{w}^k-\widetilde{\m{w}}^k\right\|_{\widetilde{G}_k}$.

\begin{lemma}\label{exten-51}
Suppose there exist $ \gamma_1>0$ and $\gamma_2>0$ with $\gamma_1 \gamma_2\geq 1$ such that
\begin{equation}\label{L1-L2}
 L_1\succeq  \gamma_1\beta B\tr B \quad \mbox{and} \quad L_2\succeq  \gamma_2\beta C\tr C.
\end{equation}
Then, for any $(\tau,s)\in\Delta$ defined in (\ref{domainK}), we have
 $\widetilde{Q}_k$ defined in (\ref{3block-tidleQk1})
is positive semidefinite and
\begin{eqnarray}
&&\left\|\m{w}^k-\widetilde{\m{w}}^k\right\|_{\widetilde{G}_k}^2\geq \left\|\m{x}^k - \m{x}^{k+1}\right\|_{\C{D}_k}^2+
\omega_0\left\|\C{K}\m{w}^{k+1}-\m{b} \right\|^2
\label{3block-impor1121}\\
 & &    +
\omega_1 \left(\left\|\C{K}\m{w}^{k+1}-\m{b} \right\|^2-\left\|\C{K}\m{w}^k-\m{b} \right\|^2 \right)\nonumber \\
&& +\omega_2\left(\left\| \left(\begin{array}{c}  \m{y}^k-\m{y}^{k+1}  \nonumber\\ \m{z}^k-\m{z}^{k+1}  \end{array} \right) \right\|_{\bar{L}}^2-\left\| \left(\begin{array}{c}  \m{y}^k-\m{y}^{k-1} \nonumber\\ \m{z}^k-\m{z}^{k-1}  \end{array} \right)\right\|_{\bar{L}}^2 \right), \nonumber
\end{eqnarray}
where $\omega_0,\omega_1, \omega_2\geq 0$ is defined in (\ref{def_omega}) and
$\bar{L}$ is defined in (\ref{bar-L}).
\end{lemma}
\begin{proof}
First, since  $L_1\succeq  \gamma_1\beta B\tr B$ and $ L_2\succeq  \gamma_2\beta C\tr C$,
it follows from $ \gamma_1>0,  \gamma_2>0$ and $\gamma_1 \gamma_2\geq 1$ that
\begin{equation}\label{Posity-BarL}
\bar{L}=\left[\begin{array}{ccc}
             L_1&& - \beta B\tr C\\
          - \beta C\tr B&&   L_2
\end{array}\right]\succeq \beta \left[\begin{array}{ccc}
             {\gamma_1} B\tr B && -B\tr C\\
          -C\tr B&&   {\gamma_2} C\tr C
\end{array}\right]\succeq \m{0}.
\end{equation}
By Lemma~\ref{Posity-1}, we have $\widetilde{Q}_k$ defined in (\ref{3block-tidleQk1})
is positive semidefinite if
\begin{equation}\label{barL-0}
\bar{L} \succeq  (\tau-1) \beta (B, C) \tr (B, C).
\end{equation}
Since $\tau \le 1$ for any $(\tau,s)\in\Delta$, we have
$ \m{0} \succeq (\tau-1) \beta (B, C) \tr (B, C)$.
Therefore, we have from (\ref{Posity-BarL}) that (\ref{barL-0}) holds
and therefore, $\widetilde{Q}_k$
defined in (\ref{3block-tidleQk1}) is positive semidefinite.
Furthermore, it follows from {Lemma}~\ref{key-lemma} that (\ref{3block-impor1121}) holds
as long as $\bar{L} \succeq \m{0}$ which is verified by (\ref{Posity-BarL}).
\end{proof}

Now, defining $\m{w}_T :=\frac{1}{T}\sum_{k=\kappa}^{\kappa+T}\tilde{\m{w}}^{k}$
for some integers $T>0$ and ${\kappa > 0}$, under the same conditions in
 {Theorem}~\ref{Sec3-theore1}, by {Lemma}~\ref{exten-51}
and similar to the proof of Theorem~\ref{Sec3-theore1},  we can obtain
  \begin{eqnarray*}\label{3block-bz6}
    &&\mathbb{E}\left[ F(\m{w}_T)-F(\m{w})+(\m{w}_T-\m{w})\tr \C{J}(\m{w})\right]\\
    &&\leq
\frac{1}{2T}\bigg\{ \sigma^2\sum\limits_{k=\kappa}^{\kappa+T} \eta_k m_k
+ \frac{4}{m_\kappa(m_\kappa+1)\eta_\kappa} \left\|\m{x}- \breve{\m{x}}^\kappa\right\|_{\C{H}}^2+\left\|\m{w}-\m{w}^\kappa\right\|^2_{\widetilde{Q}_\kappa} \nonumber\\
&&  + \omega_1 \left\|\C{K}\m{w}^\kappa-\m{b}\right\|^2+
\omega_2\left\| \left(\begin{array}{c}  \m{y}^\kappa-\m{y}^{\kappa-1} \\
\m{z}^\kappa-\m{z}^{\kappa-1}  \end{array} \right)\right\|_{\bar{L}}^2  \bigg\},
  \end{eqnarray*}
where $\omega_1 \ge 0$ and $\omega_2 \ge 0$ given in (\ref{def_omega}).
So, by the choice of the parameters $(\eta_k, m_k)$ chosen in  Theorem~\ref{2-sec-bj1},
we  can obtain
\[
\left| \mathbb{E} \big[ F(\m{w}_T)-F(\m{w}^*)\big]\right| = E_\varrho(T) =
\mathbb{E} \big[\left\|A\m{x}_{T}+B\m{y}_{T}+C\m{z}_T-\m{b} \right\|\big],
\]
where $E_\varrho(T) = \C{O} (1/T)$ for the parameter $\varrho > 1$ and
$E_\varrho(T) = \C{O} (T^{-1}\log T)$ for $\varrho = 1$.

\renewcommand\figurename{Alg.}
\begin{figure}
\begin{center}
{ \tt
\begin{tabular}{l}
\hline
\\
{\bf Parameters:}
 $\beta>0,  (\tau,s)\in \Delta$ and
$L_{i}\succeq(q-1)\beta B_i\tr B_i $ for all ${i}=1,\ldots, q$.\\
{\bf Initialization:}
$(\m{x}^0,\m{y}^0,\LAMBDA^0)$
$\in\C{X}\times\C{Y} \times \mathbb{R}^n,~ {\breve{\m{x}}^0=\m{x}^0}$.\\
For  $k=0,1,\ldots$ \\
\hspace*{.3in}Choose $m_k>0$, $\eta_k >0$ and
$ \C{M}_k $ such that $\C{M}_k - \beta A \tr A  \succeq \m{0}$.\\
\hspace*{.3in}$\m{h}^k :=$
$-A\tr \left[\LAMBDA^k-\beta( A\m{x}^k+B\m{y}^k-\m{b})\right]$. \\
\hspace*{.3in}$(\m{x}^{k+1}, \breve{\m{x}}^{k+1}) =$
{\bf xsub} ($\m{x}^k, \breve{\m{x}}^k)$ with {\bf xsub} given in ALG.\ref{algo1}. \\
\hspace*{.3in}$\LAMBDA^{k+\frac{1}{2}}=\LAMBDA^k-
\tau \beta\left(A\m{x}^{k+1}+B\m{y}^{k} -\m{b}\right)$.\\
\hspace*{.3in}For $i=1,2,\cdots,q$, \\
\hspace*{.4in}{\scriptsize$\m{y}_i^{k+1}\in\arg\min\limits_{\m{y}_i \in \C{Y}_i}  g_i(\m{y}_i)+ \frac{\beta}{2}
\left\|   A\m{x}^{k+1}+{B_i\m{y}_i+\sum\limits_{l\neq i, l=1}^{q}B_l\m{y}_l^k} -\m{b}-\frac{\LAMBDA^{k+\frac{1}{2}}}{\beta}
\right\|^2 +\frac{1}{2}\left\|\m{y}_i-\m{y}_i^k\right\|_{L_i}^2   $}.\\
\hspace*{.3in}end  \\
\hspace*{.3in}$\LAMBDA^{k+1}=\LAMBDA^{k+\frac{1}{2}}-
s \beta\left(A\m{x}^{k+1}+B\m{y}^{k+1}-\m{b}\right).$\\
end\\
\hline
\end{tabular}
}
\end{center}
\caption{ {Multi}-block extension of SAS-ADMM in partially  Jacobi update }\label{algo5}
\end{figure}
\renewcommand\figurename{Fig.}

\begin{remark}
Observing from the above analysis, Algorithm~\ref{algo3} could be in fact generalized
to Algorithm~\ref{algo5} for solving the {multi}-block separable convex optimization:
 \begin{equation} \label{Sec44-jProb}
\begin{array}{lll}
\min   & F(\m{w}):=f(\m{x})+ \sum\limits_{i=1}^{q} g_i(\m{y}_i)\\
\textrm{s.t. } &\C{K}\m{w}:=A\m{x}+\sum\limits_{i=1}^{q}B_i\m{y}_i=\m{b},\\
& \m{x}\in \C{X},\ \m{y}_i\in \C{Y}_i,\ i=1,2,\cdots,q,
\end{array}
\end{equation}
where $f$ has the same definition as in (\ref{P}),
$g_i: \C{Y}_i \to \mathbb{R} \cup \{+\infty\} $ is a convex
but possibly  nonsmooth function,
$B_i\in \mathbb{R}^{n \times n_i}$ and $\C{Y}_i \subset \mathbb{R}^{n_i} $ is a
 closed  convex  subset.
The  convergence of Algorithm~\ref{algo5} can be analogously established with proper
modifications on the convergence proof of Algorithm~\ref{algo3}.
Here, we only give a very brief explanation.
Denote $g(\m{y})=\sum\limits_{i=1}^{q} g_i(\m{y}_i), B=(B_1,\cdots,B_q),$   $\m{y}=(\m{y}_1;\cdots;\m{y}_q),$ $\m{y}^k=(\m{y}_1^k;\cdots;\m{y}_q^k)$ and $\C{Y}=\C{Y}_1\times\cdots\times\C{Y}_q$.
Then, by the first-order optimality condition  of $\m{y}_i$-subproblem, we have $\m{y}^{k+1}_i\in\C{Y}_i$ and
\begin{eqnarray*}
&&~~~~~g_i(\m{y}_i)- g_i(\m{y}^{k+1}_i)+ \left\langle \m{y}_i-\m{y}_i^{k+1},- B_i\tr \LAMBDA^{k+\frac{1}{2}} +\beta B_i\tr (\C{K}\m{w}^{k+1}-\m{b})-\right.\\
&&\qquad \quad\qquad\left.  \beta \sum\limits_{l\neq i, l=1}^{q}B_i\tr B_l( \m{y}_l^{k+1}-\m{y}_l^{k} ) +L_i(\m{y}_i^{k+1}-\m{y}_i^{k}) \right\rangle\geq  0,~ \forall\m{y}_i\in \C{Y}_i. \nonumber
\end{eqnarray*}
After adding the above inequality from $i=1$ to $q$, we can see $\m{y}^{k+1}$ satisfies
the first-order optimality condition, hence is a solution, of the following problem:
\[
\m{y}^{k+1}\in\arg\min\limits_{\m{y}\in\C{Y}}
  g(\m{y}) + \frac{1}{2}\left\|\m{y}-\m{y}^k\right\|_{\tilde{L}}^2
+ \frac{\beta}{2}
\left\| A\m{x}^{k+1}+B\m{y} -\m{b}-\frac{\LAMBDA^{k+\frac{1}{2}}}{\beta}
\right\|^2,
\]
where
 \begin{equation}\label{tilde-L}
\tilde{L}=\left[\begin{array}{cccc}
             L_1&  -\beta B_1\tr B_2&\cdots &-\beta B_1\tr B_q\\
          -\beta B_2\tr B_1&  L_2&\cdots &-\beta B_2\tr B_q\\
          \vdots&  \vdots&\ddots&\vdots\\
          -\beta B_q\tr B_1& -\beta B_q\tr B_2&\cdots &L_q\\
\end{array}\right].
\end{equation}
So,  by a similar analysis to Algorithm~\ref{algo3},   the inequality (\ref{bjc-39-141})
 holds with
{\tiny\begin{eqnarray*}
  \widetilde{Q}_k
&= &
  \left[\begin{array}{c|cccc|c}
           \C{D}_k & &  &&&\\ \hline
         &L_1+(1-\frac{\tau s}{\tau +s})\beta B_1\tr B_1& -\frac{\tau s}{\tau +s} \beta B_1\tr B_2  & \cdots&-\frac{\tau s}{\tau +s} \beta B_1\tr B_q &-\frac{\tau  }{\tau +s} B_1\tr \\
        & -\frac{\tau s}{\tau +s}\beta B_2\tr B_1& L_2+(1-\frac{\tau s}{\tau +s})\beta B_2\tr B_2  & \cdots&-\frac{\tau s}{\tau +s}\beta B_2\tr B_q &-\frac{\tau  }{\tau +s}B_2\tr\\
    & \vdots& \vdots  & \ddots&\vdots &\vdots\\
  & -\frac{\tau s}{\tau +s}\beta B_q\tr B_1&  -\frac{\tau s}{\tau +s}\beta B_q\tr B_2& \cdots& L_q+(1-\frac{\tau s}{\tau +s})\beta B_q\tr B_q &-\frac{\tau  }{\tau +s}B_q\tr\\  \hline
      &  -\frac{\tau  }{\tau +s} B_1   & -\frac{\tau  }{\tau +s} B_2  & \cdots& -\frac{\tau  }{\tau +s} B_q&  \frac{1}{(\tau +s)\beta} \m{I}
\end{array}\right]
\end{eqnarray*}}
and
{\scriptsize\begin{eqnarray*}
 \widetilde{G}_k
 &= &
\left[\begin{array}{c|cccc|c}
           \C{D}_k & &  &&&\\ \hline
          &L_1+(1-s)\beta B_1\tr B_1& -s \beta B_1\tr B_2  & \cdots&-s \beta B_1\tr B_q &(s-1)B_1\tr \\
         & -s\beta B_2\tr B_1& L_2+(1-s)\beta B_2\tr B_2  & \cdots&-s\beta B_2\tr B_q &(s-1)B_2\tr\\
    & \vdots& \vdots  & \ddots&\vdots &\vdots\\
  & -s\beta B_q\tr B_1&  -s\beta B_q\tr B_2& \cdots& L_q+(1-s)\beta B_q\tr B_q &(s-1)B_q\tr\\  \hline
      &  (s-1) B_1   & (s-1) B_2  & \cdots& (s-1) B_q&  \frac{2-\tau-s}{\beta} \m{I}
\end{array}\right].
\end{eqnarray*}}
If $ L_i\succeq (q-1)\beta   B_i\tr B_i $ for $i = 1, \ldots, q$,
then for any $(\tau,s)\in \Delta$ defined by (\ref{domainK}),  the above  matrix
$\widetilde{Q}_k$ is positive semidefinite and
\begin{eqnarray*}
&&\left\|\m{w}^k-\widetilde{\m{w}}^k\right\|_{\widetilde{G}_k}^2\geq \left\|\m{x}^k - \m{x}^{k+1}\right\|_{\C{D}_k}^2+
\omega_0\left\|\C{K}\m{w}^{k+1}-\m{b} \right\|^2\\
 & &    +
\omega_1 \left(\left\|\C{K}\m{w}^{k+1}-\m{b} \right\|^2-\left\|\C{K}\m{w}^k-\m{b} \right\|^2 \right)\nonumber \\
&& +\omega_2\left(\left\|\m{y}^{k+1}-\m{y}^k\right\|_{\tilde{L}}^2
-\left\|\m{y}^k-\m{y}^{k-1}\right\|_{\tilde{L}}^2\right),
\end{eqnarray*}
where $\omega_0,\omega_1, \omega_2\geq 0$ is defined in (\ref{def_omega}) and
$\tilde{L}$ is defined in (\ref{tilde-L}).
The above discussions imply that Algorithm \ref{algo5}  has the same convergence properties as  Algorithm  \ref{algo3}
and can be also considered as
 a stochastic extension of the deterministic GS-ADMM \cite{BLXZ2018} for solving the grouped multi-block separable convex optimization problem.
\end{remark}
\bibliographystyle{siam}

\end{document}